\documentclass[11pt,reqno]{amsart}

\DeclareMathAlphabet{\mathpzc}{OT1}{pzc}{m}{it}

\usepackage{amsmath,amsthm,amssymb,graphicx}

\usepackage{mathrsfs}

\numberwithin{equation}{section}

\theoremstyle{plain}
\newtheorem{thm}{Theorem}[section]

\newtheorem{conj*}{Conjecture}
\newtheorem{Cor}[thm]{Corollary}
\newtheorem{cor}[thm]{Corollary}

\newtheorem{prop}[thm]{Proposition}

\newtheorem{lem}[thm]{Lemma}

\theoremstyle{definition}

\theoremstyle{remark}

\newtheorem{rmk}[equation]{Remark}

\newtheorem{rem}[equation]{Remark}

\newcommand\pf{\begin{proof}}
\newcommand\epf{\end{proof}}

\newcommand\bp{\begin{pmatrix}}
\newcommand\ep{\end{pmatrix}}
\newcommand\ben{\begin{enumerate}}
\newcommand\een{\end{enumerate}}
\newcommand\be{\begin{equation}}
\newcommand\ee{\end{equation}}
\newcommand\benn{\begin{equation*}}
\newcommand\eenn{\end{equation*}}
\newcommand\bea{\begin{eqnarray}}
\newcommand\eea{\end{eqnarray}}
\newcommand\beann{\begin{eqnarray*}}
\newcommand\eeann{\end{eqnarray*}}

\newcommand{\figref}[1]{\hyperref[#1]{Figure \ref{#1}}}
\newcommand{\lemref}[1]{\hyperref[#1]{Lemma \ref{#1}}}
\newcommand{\thmref}[1]{\hyperref[#1]{Theorem \ref{#1}}}
\newcommand{\conjref}[1]{\hyperref[#1]{Conjecture \ref{#1}}}
\newcommand{\propref}[1]{\hyperref[#1]{Proposition \ref{#1}}}
\newcommand{\corref}[1]{\hyperref[#1]{Corollary \ref{#1}}}
\newcommand{\defref}[1]{\hyperref[#1]{Definition \ref{#1}}}
\newcommand{\rmkref}[1]{\hyperref[#1]{Remark \ref{#1}}}
\newcommand{\qref}[1]{\hyperref[#1]{Question \ref{#1}}}
\newcommand{\secref}[1]{\hyperref[#1]{\S\ref{#1}}}
\newcommand{\appref}[1]{\hyperref[#1]{Appendix \ref{#1}}}

\newcommand{\R}{\mathbb{R}}
\newcommand{\br}{\mathbb{R}}

\newcommand{\Z}{\mathbb{Z}}
\newcommand{\z}{\mathbb{Z}}

\newcommand{\N}{\mathbb{N}}

\newcommand{\bH}{\mathbb{H}}

\newcommand{\fh}{\frak{h}}
\newcommand{\cA}{\mathcal{A}}
\newcommand{\cB}{\mathcal{B}}
\newcommand{\cC}{\mathcal{C}}

\newcommand{\cF}{\mathcal{F}} 
\newcommand{\cG}{\mathcal{G}}

\newcommand{\cK}{\mathcal{K}}

\newcommand{\cM}{\mathcal{M}}

\newcommand{\cP}{\mathcal{P}}

\newcommand{\cS}{\mathcal{S}}

\newcommand{\gep}{\epsilon}  
\newcommand{\G}{\Gamma}      
\newcommand{\g}{\gamma}      

\newcommand{\PSL}{\operatorname{PSL}}

\newcommand{\SO}{\operatorname{SO}}

\newcommand{\Id}{\operatorname{Id}}

\newcommand{\supp}{\operatorname{supp}}

\newcommand{\vol}{\operatorname{vol}}

\renewcommand{\hat}{\widehat}

\newcommand{\dist}{\operatorname{dist}}

\renewcommand{\>}{\right\rangle}
\newcommand{\la}{\langle}
\newcommand{\ra}{\rangle}
\newcommand{\bk}{\backslash}
\newcommand{\ba}{\backslash}

\newcommand{\op}{\operatorname}
\newcommand{\e}{\gep}

\theoremstyle{plain}

\numberwithin{equation}{section}

\newcommand{\bms}{\mathsf m}
\newcommand{\m}{\mathsf m}

\newcommand{\BR}{\rm{BR}}
\newcommand{\PS}{\rm{PS}}
\newcommand{\Leb}{\ell}
\newcommand{\Haar}{\rm{Haar}}

\renewcommand{\P}{\mathcal P}
\newcommand{\T}{\operatorname{T}}
\newcommand{\bc}{\mathbb C}

\usepackage{caption}
\usepackage[labelfont=rm]{subcaption}

\usepackage[pagebackref=true, colorlinks]{hyperref}

\hypersetup{pdffitwindow=true,linkcolor=blue,citecolor=blue,urlcolor=blue,menucolor=blue}

\usepackage{comment}

\begin{document}
\title[]{
Shrinking targets for the geodesic flow on geometrically finite hyperbolic manifolds}
\author{Dubi Kelmer}
\author{Hee Oh}
\email{hee.oh@yale.edu}
\address{Yale University 
New Haven, CT}

\thanks{Kelmer is partially supported by NSF CAREER grant DMS-1651563 and Oh is partially supported by NSF grants.}
\email{kelmer@bc.edu}
\address{Boston College,  Chestnut Hill, MA}


\date{\today}%
\dedicatory{}%
\commby{}%
\begin{abstract}
Let $\cM$ be a geometrically finite hyperbolic manifold.
We present a very general theorem on the shrinking target problem for the geodesic flow, using its
 exponential mixing. This includes a strengthening of Sullivan's logarithm law for the excursion rate of the geodesic flow. More generally,  
 we prove logarithm laws for the first hitting time for shrinking cusp neighborhoods, shrinking tubular neighborhoods of a closed geodesic,
 and shrinking metric balls, as well as give quantitative estimates for the time a generic geodesic spends in such shrinking targets. 
\end{abstract}

 \maketitle
 
 \section{Introduction}

 Let $\cM$ be a complete hyperbolic manifold of dimension $n\ge 2$.  Denote by $\mathcal G^t$ the geodesic flow on the unit tangent bundle $\T^1(\cM)$. 
 If $\cM$ is of finite volume, but non-compact, Sullivan \cite{Sul82} showed in 1982 the following logarithm law
 for the rate of the excursion of the geodesic flow: for any $o\in \cM$, and for almost all $x\in \T^1(\cM)$,  
\begin{equation}\label{c1} \limsup_{t\to \infty} \frac{d(\mathcal G^t(x), o)}{\log t} =\frac{1}{n-1}\end{equation}
 where $d(\mathcal G^t(x), o)$ is the hyperbolic distance between the basepoint of $\mathcal G^t(x)$ and  $o$.  
 
  This result can be viewed as a special case of the so-called shrinking target problem for the geodesic flow, which asks the behavior
 of a generic geodesic ray with respect to a given sequence of shrinking subsets. Indeed,
 if we consider the family of shrinking cuspidal neighborhoods $\mathfrak h_t:=\{z\in \cM: d(o, z)>t\}$, $t> 1$,
 then  \eqref{c1} is equivalent to the following logarithm law for the first hitting time: for almost all $x$,
\begin{equation}\label{cc2} \liminf_{t\to \infty} \frac{
\log \tau_{ \mathfrak h_t} (x) }{ t} ={n-1} \end{equation}
where
  $\tau_{\mathfrak h_t}(x):= \inf\{s>0: \mathcal G^s(x)\in \mathfrak h_t \}$.

In this paper, we investigate shrinking target problems for the geodesic flow on a geometrically finite hyperbolic manifold
$\cM$, and prove results which are far reaching strengthening and generalizations of \eqref{cc2}, and hence of \eqref{c1}.

Let $\bH^n$ denote the $n$-dimensional hyperbolic space and let $G:=\op{Isom}^+(\bH^n)$ be the group of all orientation preserving isometries. We may present a  complete hyperbolic manifold $\cM$ as the quotient $\Gamma\ba \bH^n$ where
 $\Gamma$  is a torsion-free discrete subgroup of $G$. We assume that $\Gamma$ is
   Zariski dense  and geometrically finite in the whole paper.
Denote by $\Lambda\subset \partial \bH^n$ the limit set of $\Gamma$ and by $0<\delta\le n-1$ the critical exponent of $\Gamma$. The maximal entropy of the geodesic flow
on $\T^1(\cM)$ is given by $\delta$, and there exists a unique ergodic probability measure of maximal entropy, called the Bowen-Margulis-Sullivan measure, which we denote by $\mathsf m$.
The support of $\m$ is precisely the non-wandering set for the geodesic flow  and hence the shrinking target problem in this setting
is interesting only for those shrinking subsets in the support of $\m$ and for $\m$-almost all points. 
 Now since $\mathcal G^t$ is ergodic for $\mathsf m$, the Birkhoff ergodic theorem 
  says that for a given Borel subset $B\subset \T^1(\cM)$, we have the following for $\m$-almost all $x\in \T^1(\cM)$,
 \begin{equation} \label{erg} \lim_{t\to \infty} \frac{\Leb\{0<s<t: \mathcal G^s(x)\in B\}}{t} = \m (B)\end{equation}
 where $\Leb$ denotes the Lebesgue measure on $\br$.
The shrinking target problem asks a finer question on the
set of times $\{s>0: \mathcal G^s(x)\in B_t\}$ for a given family $\{B_t\}$ of shrinking sets and for $\m$-a.e. $x$.
The three main questions we address in this paper for $\m$-a.e. $x\in \T^1(\cM)$ are as follows:

\begin{enumerate}

\item (Logarithm laws) Is there a logarithm law for the first hitting time 
\begin{equation}\label{tty}\tau_{B_t}(x):=\inf\{s>0: \mathcal G^s(x) \in B_t\}?\end{equation}  

\item (Shrinking rate threshold) How fast can $B_t$ shrink so that $$ \tau_{B_t}(x) < t $$ for an infinite sequence of times $t$ tending to $\infty$ or for all sufficiently large $t\gg 1$?

\item (Quantitative estimates) How fast can $B_t$ shrink so that\footnote{The notation $f_t\ll g_t$ means that for all $t>1$, $f_t \leq c\, g_t$ for some absolute constant $c>0$, and
 we write $ f_t \asymp g_t$ if $f_t\ll g_t$ and $g_t\ll f_t$. We sometimes indicate the dependence of the implied constant in subscripts.}
 $$ \Leb \{0<s<t: \mathcal G^s(x)\in B_t\} \asymp  t  \cdot \m (B_t)$$ for an
 infinite sequence of times $t$ tending to $\infty$,
or for all sufficiently large $t\gg 1$?

\end{enumerate}
 

 In order to address the above questions, we need to impose
  certain regularity conditions on the shrinking targets.  Let  $K< G$  be a maximal compact subgroup and identify $\cM$ with $\Gamma\ba G/K$.
There exists  a one parameter diagonalizable  subgroup $A=\{a_t\}$ so that if $M$ denotes the centralizer of $A$ in $K$, then
 the unit tangent bundle $\T^1(\cM)$ can be identified with $\Gamma\ba G/M$ in the way that
the geodesic flow $\mathcal G^t$ on $\T^1(\cM)$ corresponds to the right translation action of $a_t$ on $\Gamma\ba G/M$. 
We fix $\ell \gg \op{dim}(\cM)$ and
the Sobolev norm $\mathcal S=\mathcal S_{\infty, \ell}$ on $C^\infty(\Gamma\ba G)$ given by
$$\mathcal S(\Psi)=\sum \|D(\Psi)\|_\infty$$ where
the sum is taken over all monomials in a fixed basis of $\op{Lie}(G)$ of order at most $\ell$.

 A family of shrinking targets in  $\T^1(\cM)$ means a collection $\cB=\{B_t\subset \T^1(\cM): t>1\}$ such that
$\m(B_t)>0$,  $B_t\supset B_s$ for $s>t$,  and $\lim_{t\to \infty} \mathsf m(B_t)= 0$.
A family $\{B_t\}$  of shrinking targets is said to be {\it inner regular} (resp. {\it outer regular}) 
if there exist $\alpha>0$ and a family of functions $ \Psi_t^-\in C^\infty(\T^1(\cM))$ (resp. $ \Psi_t^+\in C^\infty(\T^1(\cM)$))  such that
\begin{itemize}
\item  $0\leq \Psi_t^-\le  \Id_{B_t}$ (resp. $\Id_{B_t}\le  \Psi_t^+\ll 1$); 
\item
$\bms(B_t)\ll \bms(\Psi_t^-)$ (resp. $\bms(\Psi_t^+)\ll \bms(B_t)$);
\item $\cS(\Psi_t^\pm)\ll
\bms(B_t)^{-\alpha}$
\end{itemize}
where the implied constants are independent of $t$.
A family $\{B_t\}$ is said to be {\it regular} if it is both inner and outer regular.

 We note that this regularity condition is rather mild, and is satisfied by most families of naturally occurring shrinking targets. 
Such examples include shrinking cusp neighborhoods,  shrinking tubular neighborhoods of a closed geodesic and shrinking metric balls, as will be shown later.

{ In the rest of the introduction, we assume  that $\mathcal B=\{B_t: t\gg 1\}$ is a family of shrinking targets in $\T^1(\cM)$.}


\subsection{Logarithm laws}
{ For discrete time dynamical systems, it is expected that the 
first hitting time would be inversely proportional to the measure of the shrinking target; it is indeed the case for the
discretized geodesic flow. For the continuous geodesic flow,  it turns out that it is inversely proportional  to the measure of a thickened set $\tilde B_t :=\cup_{|s|<1/2} \mathcal G^{s}(B_t)$: 

 \begin{thm}\label{t:loglaw}\label{t6}
\begin{enumerate}
\item If $\{B_t\}$ is inner regular, 
then
$$\lim_{t\to \infty} \frac{\log(\tau^d_{B_t}(x))}{-\log(\bms(B_t))}=1 \quad \mbox{ for $\bms$-a.e. $x\in \T^1(\cM)$}.$$
where $\tau^d_{B}(x)=\min\{n\in \N: \mathcal G^n x\in B\}$. 

\item If  $\{\tilde B_t\}$ is inner regular, then
$$\lim_{t\to \infty} \frac{\log(\tau_{B_t}(x))}{-\log(\m (\tilde{B_t}))}=1 \quad \mbox{ for $\bms$-a.e. $x\in \T^1(\cM)$}. $$
\end{enumerate}
\end{thm}}

\begin{rmk} 
When $|\log \m(\tilde B_t) | \asymp |\log \m (B_t)|$, the first hitting time for the discrete flow  $\{\mathcal G^n: n\in \N\}$ behaves in the same way for the continuous flow.
This is indeed the case for shrinking cusp neighborhoods or tubular neighborhoods of a closed geodesic. However, there are also cases 
when $|\log \m(\tilde B_t)|$ is much larger than $|\log \m (B_t)|$, such as the case of shrinking metric balls.
\end{rmk}

We note that  logarithm laws for the first hitting time were studied for certain families of shrinking targets in many examples of 
discrete time dynamical systems with fast mixing, see e.g. \cite{Galatolo07, Galatolo10, GalatoloNisoli11}.

\subsection{Shrinking rate threshold} 
In order to ensure that a generic orbit $\mathcal G^s(x)$ hits $B_t$ before time $t$  for an infinite sequence of $t$ tending to $\infty$,  the easy half of the Borel-Canteli lemma implies that it is necessary to have
$\sum_k \bms(\tilde{B}_k)=\infty$, from which $\limsup_{t\to \infty} {\log^2(t) t \cdot \bms(\tilde{B}_{t})} =\infty$ follows. The first part of the following theorem says that this condition is also sufficient, up to logarithmic factors.  The second part says that a generic orbit $\mathcal G^s(x)$ hits $B_t$  before time $t$, for all sufficiently large $t$, under a slightly stronger assumption on the rate of shrinking (see Theorem \ref{inl}). 
\begin{thm} \label{t:thresholdC}\label{t7}
Suppose that $\{\tilde B_t\}$ is inner regular.

\begin{enumerate}
\item If  $\limsup_{t\to\infty}\tfrac{t\bms(\tilde{B}_{t})}{|\log(\bms(\tilde{B}_{t}))|}=\infty$, then 
 $$\liminf_{t\to \infty} \frac{\tau_{B_t}(x)}{t}\le 1 \quad\text{for $\bms$-a.e. $x\in \T^1(\cM)$.}$$
 
 \item If 
$\sum_{j=1}^\infty \frac{|\log(\bms(\tilde{B}_{t_j}))|}{t_j\bms(\tilde{B}_{t_j})}<\infty$ for some sequence $t_j\to \infty$,
then
  $$\limsup_{t\to \infty} \frac{\tau_{B_t}(x)}{t}\le 1 \quad\text{ for $\bms$-a.e. $x\in \T^1(\cM)$.}$$
\end{enumerate}
\end{thm}

\subsection{Quantitative estimates}
In order to answer a more refined question regarding the amount of time that a generic geodesic ray spends in a shrinking target,
we require our family of targets to be regular and  their measures do not change too fast in the sense that $\bms(B_t)\asymp \bms(B_{2t})$.

With these additional regularity assumptions, we have the following (see Theorem \ref{qua} below for a more general result).

\begin{thm}\label{t:asymp}\label{t8}
Suppose that $\{B_t\}$ is regular and that $\m (B_{2t})\asymp \m (B_t)$.
\begin{enumerate}
\item If $\limsup_{t\to\infty} \frac{t\bms(B_{t})}{|\log(\bms(B_{t}))|}= \infty$, then there exists a sequence $t_k\to \infty$ such that for $\bms$-a.e. $x$,
$$ \frac{\Leb \{0 < s< t_k:\mathcal G^s(x) \in B_{t_k}\} }{t_k}\asymp \bms(B_{t_k}).$$

\item If
$\sum_{j=1}^\infty \tfrac{|\log(\bms(B_{2^j}))|}{2^j\bms(B_{2^j})}<\infty ,$
then for $\bms$-a.e. $x$, 
$$ \frac{\Leb \{0<s<t : \mathcal G^s(x) \in B_{t}\}}{t}\asymp \bms(B_t).$$
\end{enumerate}
\end{thm}

We observe that unlike Theorems \ref{t6} and \ref{t7}, the amount of time that the geodesic flow spends in the targets is governed by the measure of the original targets rather than by their thickenings.

\begin{rem}\begin{enumerate}
\item We note that in many examples the measure of the shrinking targets decay like $\m(B_t)\asymp t^{-\eta}$ for some $\eta>0$. In such cases, we have $\m(B_t)\asymp \m(B_{2t})$ and the rest of the conditions of Theorems \ref{t7} and \ref{t8} are satisfied if $\eta<1$.

\item As mentioned before, the extra conditions on the rate of decay we have in Theorems \ref{t7} and \ref{t8} are sharp, but up to logarithmic factors. While it would be very interesting to have sharp conditions on the nose, we note that such a result is notoriously hard. 
Even when $\cM$ has finite volume, sharp results regarding Theorem \ref{t8}(1) are  known only in some very special cases when the shrinking targets are cusp neighborhoods \cite{Sul82}, or spherical balls \cite{Mau06} (or general spherical targets if one considers discrete time dynamics \cite{Kel17}). There are no known sharp results regarding Theorem \ref{t8}(2). We refer to \cite{KKR19} where this kind of problem is studied for systems with almost perfect mixing. 

\item All the results described above still hold as stated if we replace the unit tangent bundle $\T^1(\cM)$ with the frame bundle  $\G\bk G$,
provided  $\delta>n-2$. 
We note if $\cM$ contains a co-dimension one properly immersed totally geodesic sub-manifold of finite volume, then $\delta>n-2$, so this stronger condition still holds in many examples.
\end{enumerate}
\end{rem}

For some concrete applications of these results, we discuss three families of shrinking targets to which our theorems apply. In order to define these families, we fix a left $G$-invariant and right $K$-invariant metric $d$ on $G$ which descends to the hyperbolic metric on $\bH^n=G/K$. This metric then naturally defines a distance function, $\dist(\cdot,\cdot)$ on $\T^1(\cM)=\G\bk G/M$.

 \subsection{Cusp excursion}
 The convex core of $\cM$ is defined by $\op{core}(\cM)=\Gamma\ba \op{hull}(\Lambda)$, where $\op{hull}(\Lambda)$ defines the convex hull of the limit set $\Lambda$.
As $\cM$ is geometrically finite, there are finitely many disjoint cuspidal regions 
whose complement in $\op{core}(\cM)$
is a compact submanifold.   Let $\mathfrak h_i$, $1\le i\le k$, denote the pre-images in $\T^1(\cM)$ of these cuspidal regions under the base point projection $\pi: \T^1(\cM)\to \cM$.
For each $i$, we denote by $\kappa_i$ the rank of $\mathfrak h_i$, that is, the rank of the
maximal free abelian subgroup of  the stabilizer $\op{Stab}_\Gamma(\mathfrak h_i)$. It is known that $ \kappa_i <2\delta$.

For each $i$ and $t> 1$, consider the following cusp neighborhood \begin{equation}\label{e:cuspnhood}
\mathfrak h_{i,t}:=\{x\in \mathfrak h_i: \dist(x, \partial \mathfrak h_i )>t\}.\end{equation}

For each $i$, we show that the shrinking family  $\{\mathfrak h_{i,t}: t>1\}$
is regular and that \begin{equation}\label{e:cuspmeasure}
\bms(\fh_{i,t})\asymp e^{-(2\delta-\kappa_i)t}.
\end{equation}  
(see section \ref{s:cusp}).
Applying our results to this family, we get the following:
 \begin{thm} \label{t1} Fix $1\le i\le k$.  
\begin{enumerate} 
\item   For $\bms$-a.e. $x\in \T^1(\cM)$, 
 \begin{equation*} \lim_{t\to \infty} 
 \frac{\log \tau_{\mathfrak h_{i,t}}(x)}{ t} ={2 \delta -\kappa_i}. \end{equation*}
\item For any $0<\eta<\frac{1}{2\delta -\kappa_i}$, 
and for $\bms$-a.e. $x\in \T^1(\cM)$, 
$$\Leb\{0<s<t: \mathcal G^s(x)\in \mathfrak h_{i,\eta \log t} \} \asymp t^{1-\eta(2\delta -\kappa_i)} .$$
\end{enumerate} 
 \end{thm}

{
\begin{rem} As mentioned before, it is not hard to show that 
\begin{equation}\label{sl} \liminf_{t\to \infty} \frac{\log (\tau_{\fh_{t}}(x))}{t}=\left(\limsup_{t\to\infty}\frac{\dist(\mathcal G^t(x), o)}{\log t }\right)^{-1}\end{equation}
where $\fh_t=\bigcup_{1\le i\le k} \fh_{i,t}$.
Stratmann and Velani  showed that \eqref{sl} is equal to
$2\delta-\max_i \kappa_i$ \cite{SV95}, and hence extended Sullivan's logarithm law \eqref{c1} to geometrically finite manifolds.
 Theorem \ref{t1}(1) presents a stronger version, as 
 we consider excursion to individual cusps as well as obtain an actual limit rather than $\lim\inf$. 
\end{rem}}

For the sake of a concrete application, we give a reformulation of Theorem \ref{t1}(1) in the case of Apollonian manifolds.
An Apollonian gasket $\P=\bigcup C_i$ is a countable union of circles obtained by repeatedly inscribing circles into the triangular interstices of four mutually
tangent circles with disjoint interiors in the complex plane (where lines are considered as circles). The symmetry group
$\{g\in \PSL_2(\bc): g(\P)=\P\}$ is a discrete subgroup of $\PSL_2(\bc)$ which acts on $\hat \bc$ by M\"obius transformations and its torsion-free subgroup of finite index is called an Apollonian group, which we denote by $\Gamma$.  Via the Poincar\'e extension theorem, we can identify $\PSL_2(\bc)$ with $\op{Isom}^+(\bH^3)$
for the upper-half space model $\bH^3$ of the hyperbolic space. The quotient manifold $\Gamma\ba \bH^3$ is called an Apollonian manifold,
which is known to be geometrically finite with all cusps having rank one. Its limit set is equal to
the closure $\overline{\P}$, and supports a locally finite Hausdorff measure $\mathcal H$ of  dimension $\delta=1.30568(8)$ \cite{Mc}. 

Fix a tangent point $\xi= C_i\cap C_j$ for $i\ne j$ and consider a sufficiently small Euclidean ball $B$ in $\bH^3$ based at $\xi$,
so that $\mathcal B=\Gamma (B)$ is a disjoint collection of Euclidean balls.

Fix $o\in \bH^3$ outside of $B$,
let $B(t)\subset B$ be the Euclidean ball based at $\xi$ and $d_{\bH^3}(o, B(t))= d_{\bH^3}(o, B)+t$. Set $\mathcal B_t:= \Gamma (B(t))$.

The following is a consequence of Theorem \ref{t1}:
\begin{cor} Let $\P$ be an Apollonian gasket. For $\mathcal H$-almost all initial  direction $v$ toward $\overline \P$,
\begin{equation}\label{c3} \lim_{t\to \infty} 
\frac{\log  ({\inf \{s>0: v_s\in \mathcal B_t\} )} }{ t} =2\delta -1 (=1.6113... )
 \end{equation}
where $v_s$ denotes the base point of $\mathcal G^s(v)$.
\end{cor}

\subsection{Tubular neighborhoods}
Another natural family of shrinking targets is given by tubular neighborhoods of a closed geodesic. 
 For a closed geodesic $\mathcal C\subset \T^1(\cM)$ and $\e>0$, we consider the $\e$-tubular neighborhood of $\mathcal C$:
  $$\mathcal C_\e:=\left\{x\in \T^1(\cM): \dist(x, \mathcal C)\le \e\right\}.$$
The family $\{\mathcal C_{1/t}: t> 1\}$ forms a family of shrinking neighborhoods of $\mathcal C$.
We show that $\{\cC_{1/t}: t> 1\}$ is a regular family with $\bms(\cC_{1/t})\asymp \bms(\tilde \cC_{1/t})\asymp t^{-2\delta}$.
 Applying our results to this family of shrinking targets gives the following result on the amount of time a generic geodesic spirals near a fixed closed geodesic (cf. \cite[Theorem 1.1]{HP10} for a similar result in a negatively curved compact manifold).

 \begin{thm} \label{t55} Let $\mathcal C\subset \T^1(\cM)$ be a closed geodesic.  Then for $\mathsf m$-a.e. $x\in \T^1(\cM)$,
we have the following:

\begin{enumerate}
\item  \begin{equation*} 
\lim_{t\to \infty} \frac{\log \tau_{\mathcal C_{1/t}}(x) }{\log t} = {2 \delta}; \end{equation*}
 
 \item  For any $0<\eta<\tfrac{1}{2\delta}$ and for all $t>1$,
   $$ \Leb\{0< s <t: \dist(\mathcal G^s(x),\cC)<t^{-\eta}\} \;\; \asymp \;\;   t^{1-2\delta \eta}.$$
\end{enumerate}
 \end{thm}

\begin{rem}  
 Since for any $x\in \T^1(\cM)$  we have that 
 $$\liminf_{t\to \infty} \frac{\log(\tau_{\cC_{1/t}}(x))}{\log t}=\left(\limsup_{t\to\infty}\frac{-\log(\dist(\mathcal G^t(x),\cC))}{\log t} \right)^{-1},$$
Theorem \ref{t55} (1) implies that for $\bms$-a.e. $x\in \T^1(\cM)$,
 \begin{equation}\label{dmp}  \limsup_{t\to\infty}\frac{-\log( \dist(\mathcal G^t(x),\cC))}{\log t}=\frac{1}{2\delta},\end{equation}
which was previously shown in \cite[Theorem 4]{DMPV95} to hold for the special case of convex co-compact hyperbolic {\it surfaces}.  

 \end{rem}

\subsection{Shrinking balls}
For any fixed $x_0\in \op{supp}(\m)$, 
we show that the family of shrinking metric balls $B_{t}(x_0):=\{x\in \T^1(\cM): \dist(x,x_0)<1/t\}$
 is regular and  satisfies $\bms(B_{t}(x_0))\asymp \bms(B_{2t}(x_0))$.
When $\G$ is convex co-compact,
$\bms(B_{t}(x_0))\asymp t^{-(2\delta+1)}$ and $\bms(\tilde{B}_{t}(x_0))\asymp t^{-2\delta}$ (see \S \ref{sball}). In particular our results imply the following:

 \begin{thm} \label{t3} Let $\cM$ be convex cocompact.  Fix $x_0\in\op{supp}(\m)$.
  Then for $\mathsf m$-a.e. $x\in \T^1(\cM)$,
\begin{enumerate}
\item  \begin{equation}\label{c5} 
\lim_{t\to \infty} \frac{\log \tau_{B_{t}(x_0)} (x)}{\log t} = {2 \delta}. \end{equation}
  \item For $0<\eta< \frac{1}{2\delta+1}$, we have
   $$ \Leb \{0 <s < t: \dist(\mathcal G^s(x), x_0) \le {t^{\eta}} \} \;\; \asymp \;\;   t^{1-(2\delta+1) \eta}.$$
\end{enumerate}
\end{thm}
 
 When $\cM$ has cusps, the situation is more complicated as $\bms(B_{t}(x_0))$ can fluctuate, with the fluctuation depending on $x_0$ (or more precisely on the cusp excursions of the geodesic emanating from $x_0 \in \T^1(\cM)$). Combining our previous results on cusp excursions, we can show the following 
 \begin{thm}\label{t4}
Suppose that $\cM$ has cusps. 
\begin{enumerate}
\item For $\bms$-a.e. $x_0\in \T^1(\cM)$, and for $\m$-a.e. $x\in \T^1(\cM)$,
$$\lim_{t\to \infty} \frac{\log \tau_{B_{t}(x_0)} (x)}{\log t} = {2 \delta} .$$
\item For any pair of distinct cusps of ranks $\kappa_1,\kappa_2$, we can find $x_0\in  \T^1(\cM)$ such that  
 for $\m$-a.e. $x\in \T^1(\cM)$,
$$
\lim_{t\to \infty}\frac{\log(\tau_{B_{t}(x_0)}(x))}{\log t}=4\delta-\kappa_1-\kappa_2.$$
\end{enumerate}
 \end{thm} 

\begin{rem}
We note that if $\cM$ has finite volume, then $\delta=n-1$ and  $\m(\tilde{B}_{t}(x_0))\asymp t^{-(2n-2)}$. Hence, in this case,  the same arguments imply that for $\m$-a.e. $x\in \T^1(\cM)$, we have $\lim_{t\to \infty} \frac{\log \tau_{B_{t}(x_0)} (x)}{\log t} = 2(n-1)$. We note that here the shrinking targets are in $\T^1(\cM)$, unlike the results of \cite{Mau06} which considered shrinking balls inside $\cM$, in which case the limit is $n-1$ (see also \cite{KZ17}, for related result for the discrete time geodesic flow).
\end{rem}

 \subsection{Strategy of proof} 

First  we define an averaging operator, along the discrete time, acting on  $L^2(\T^1(\cM),\m)$:
 $$\lambda_T(\Psi)(x)=\frac{1}{T} \sum_{k=1}^T \Psi(\cG^k(x)).$$
 If $\Psi$ is the characteristic function of $B$, we simply write $\lambda_T(B)$ instead of $\lambda_T(1_B)$.
 The Birkhoff  ergodic theorem implies that for a.e. $x\in X$,
 $$\lim_{T\to \infty} \lambda_T(\Psi)(x) =\int_{\T^1(\cM)}\Psi d\bms.$$
 
We note that if we had  a rate control in this convergence such as 
 \begin{equation}\label{control} |\lambda_T(B_t)(x) -\m(B_t)|\ll \frac{ \sqrt{\m(B_t)} |\log (\m(B_t))|}{\sqrt{T}}, \end{equation}
 we would get
 \begin{equation}\label{tau1} \log \tau^d_{B_t}(x) \le |\log \m(B_t) | +2\log|\log \m(B_t)| \end{equation}
 just from the simple observation that $\lambda_{\tau^d_{B_t}(x)}(B_t)=0$.

An estimate like \eqref{control} is too strong to be true for a.e. individual points $x$. So, instead, we prove its mean-version for all smooth functions $\Psi\in L^2(\T^1(\cM),\m)$,  that is, 
\begin{equation}\label{control2} \|\lambda_T(\Psi)-\m(\Psi)\|_2 \leq C \frac{\|\Psi\| \log (\frac{\mathcal S(\Psi)}{ \|\Psi\|_2}) }{\sqrt{T}} \end{equation}
for some uniform constant $C>0$.
The regularity conditions imposed on the thickenings $\tilde B_t$ of our shrinking targets are precisely so that
we could apply \eqref{control2} to smooth functions which approximates $1_{\tilde B_t}$ and deduce
\begin{equation}\label{control3} \|\lambda_T(\tilde B_t)-\m(\tilde B_t)\|_2 \ll   \frac{\sqrt{\m(\tilde B_t)} \log |(\m(\tilde B_t))| }{\sqrt{T}} .\end{equation}

This effective mean ergodic theorem for $\tilde B_t$'s enables us to obtain 
that for a.e. $x$,
 \begin{equation}\label{tau1} \log \tau^d_{\tilde B_t}(x) \le | \log \m(\tilde B_t)|+O(\log|\log \m(\tilde B_t)|), \end{equation}
for all sufficiently large $t$.
Using that $|\tau^d_{\tilde B_t}(x)-\tau_{B_t}(x)|\le 1$, we deduce that 
$$\limsup_{t\to \infty} \frac{\log \tau_{B_t}(x)}{-\log \m(\tilde B_t)}\le 1 .$$
This is the non-trivial direction of the logarithm law Theorem \ref{t6}; the other direction holds for general shrinking targets in any dynamical system (see e.g. \cite[Lemma 2.2]{KelmerYu17}). Theorems \ref{t7} and \ref{t8} are also proved in a similar spirit using the effective mean ergodic theorem. 

 The use of quantitative mixing of geodesic flow in the shrinking target problem in the homogeneous setting goes back to the work of Kleinbock and Margulis \cite{KM99}, and
the idea of using an effective mean ergodic theorem was first introduced in  \cite{GK17} and more explicitly in \cite{Kel17, KelmerYu17}, where these ideas were used to prove the analogous results for finite volume hyperbolic manifold.

Here we will use the following exponential decay of matrix coefficients for geometrically finite hyperbolic manifolds:
\begin{thm}\label{mix}
There exists $\eta_0>0$ such that 
for any $\Psi_1, \Psi_2\in C^\infty(\T^1(\cM))$ with support in one-neighborhood of $\op{supp}(\m)$,  for all $t\geq 1$,
\be\label{mmix0}  \int_{\T^1(\cM)} \Psi_1(\mathcal G^t(x)) \Psi_2(x) \; d\m (x)
 = \m(\Psi_1)\m (\Psi_2) +O( e^{-\eta_0 t}   \cS (\Psi_1) \cS(\Psi_2)). \ee
{ Moreover, $\eta_0$ is explicitly computable when $\delta>\frac{n-1}{2}$, depending only on the spectral gap for the Laplacian on $L^2(\cM)$.
 If $\Gamma$ is convex cocompact or $\delta>n-2$, \eqref{mmix0} with $\m$ replaced by its $M$-invariant lift on $\Gamma\ba G$
  holds for any $\Psi_1, \Psi_2\in C^\infty(\Gamma\ba G)$.}
 \end{thm}
 This theorem was obtained in (\cite{MO15}, \cite{EO}) 
for compactly supported functions under the assumption $\delta>\frac{n-1}{2}$
and in  \cite{Sto11} for any convex cocompact $\Gamma$ (see also \cite{Wi} for the same result for the frame flow).
 In order to study shrinking target problem
for cusp neighborhoods as described in Theorem \ref{t1}, removing the compact support condition is crucial as we need to study functions that are positive
on cusps.  We use the quantitative decay of the matrix coefficient of the functions
 $L^2(\Gamma\ba G)$ with respect to the Haar measure $\m^{\Haar}$ in \cite{MO15}, and exploit the product structures of $\m$ and $\m^{\Haar}$ to transfer the exponential rate information on
  the transversal intersections of $\mathcal G^t (B_\epsilon(x))$ for the flow box $B_\e(x)$, that we get from the behavior of the
  correlation function with respect to $\m^{\Haar}$, to the behavior of the correlation function with respect to $\m$.
Here $\e$ depends on the injectivity radius of $x$, and as we need to control the exponential rate independent of
the injectivity radius for Theorem \ref{mix}, which is required to deal with functions which are not compactly supported, the whole procedure
turns out to be technically quite subtle.
{ The remaining cases of geometrically finite manifolds with cusps are proved in a recent work of Li-Pan \cite{LiPan20}.}



\medskip

After some preliminaries given in section \ref{s:preliminaries}, we devote section \ref{s:mix} to the proof of Theorem \ref{mix}.
With this result in hand, we prove effective mean ergodic theorem in this setting (see Theorem \ref{t:MET}), and use it in section \ref{s:shrinking} to establish results on shrinking target problems for both the discrete and continuous time flow. While the results we obtain for the discrete time flow are essentially optimal, this is not the case for some of the results for continuous time flow. Nevertheless, in section \ref{s:cont}, we show how one can obtain optimal results for the continuous flow by translating it into a discrete time flow problem for a thickened target.  In section \ref{s:example}, we deduce Theorems \ref{t1}, \ref{t55}, \ref{t3} and \ref{t4} by proving the regularity of the corresponding shrinking sets and by computing their volumes using Sullivan's shadow lemma and the structure of cusps
for geometrically finite manifolds.

 \section{Preliminaries and notation}\label{s:preliminaries}
 
\subsection{Notations and conventions} 
Let $G\cong \SO(n,1)^o$ be the group of orientation preserving isometries of $\bH^n$, and $\G < G$ 
 a geometrically finite, torsion-free, Zariski dense, discrete subgroup of $G$. 
 We denote by $\Lambda$ the limit set of $\Gamma$, and by $0<\delta\le n-1$
 the Hausdorff dimension of $\Lambda$, which is equal to the critical exponent of $\Gamma$.
 Let $\cM=\Gamma\ba \bH^n$.
Let  $K< G$  be a maximal compact subgroup and identify $\cM$ with $\Gamma\ba G/K$.
There exists  a one parameter diagonalizable  subgroup $A=\{a_t\}$ so that if $M$ denotes the centralizer of $A$ in $K$, then
 the unit tangent bundle $\T^1(\cM)$ can be identified with $\Gamma\ba G/M$ in the way that
the geodesic flow $\mathcal G^t$ on $\T^1(\cM)$ corresponds to the right translation action of $a_t$ on $\Gamma\ba G/M$. 
  With this identification we can work in the homogeneous space $\G\bk G$ and  think of subsets and functions on $\T^1(\cM)$  and $\cM$ respectively as $M$-invariant (resp. $K$ invariant) subsets and functions on $\G\bk G$. 


We say that two families $\{B_t\}$ and $\{A_t\}$ of shrinking sets are Lipschitz equivalent and write $B_t\asymp A_t$, if there are some positive constants $c_1,c_2$ such that $B_{c_1t}\subseteq A_{t}\subseteq B_{c_2 t}$ for all $t>1$.

We fix a left $G$-invariant and right $K$-invariant metric $d$ on $G$ which descends to the hyperbolic metric on $\bH^n=G/K$.
This induces a unique metric on $G/M$ which we will also denote by $d$ by abuse of notation. The metric $d$ defines a distance function on $\T^1(\cM)=\G\bk G/M$  given by
$\dist(\G g,\G h)=\inf_{\g\in \G} d(\g g,h)$. 

\subsection{Invariant measures}
For $\xi\in \partial\bH^n$, let $\beta_\xi:\bH^n\times\bH^n\to \R$ denote the Busemann function for the geodesic flow, defined by 
$$\beta_\xi(x,y)=\lim_{t\to\infty} d(x,\xi(t))-d(y,\xi(t)),$$
with $\xi(t)$ a unit speed geodesic ray toward $\xi$.
 A family of measures
$\{\mu_x:x\in \bH^n\}$ is called  a {\em $\G$-invariant conformal
density\/} of dimension $\delta_\mu > 0$  on $\partial\bH^n$, if  each
$\mu_x$ is a non-zero finite Borel measure on $\partial\bH^n$
satisfying for any $x,y\in \bH^n$, $\xi\in \partial\bH^n$ and
$\gamma\in \G$,
$$\gamma_*\mu_x=\mu_{\gamma x}\quad\text{and}\quad
 \frac{d\mu_y}{d\mu_x}(\xi)=e^{-\delta_\mu \beta_{\xi} (y,x)}, $$
where $\gamma_*\mu_x(F)=\mu_x(\gamma^{-1}(F))$ for any Borel
subset $F$ of $\partial\bH^n$. 

In particular, the {\em{Patterson-Sullivan density}} $\{\nu_x\}$ is a $\G$-invariant conformal density supported on the limit set $\Lambda$ 
of dimension $\delta$ 
and the {\em{Lebesgue density}} $\{m_x\}$ is a $G$-invariant conformal density of dimension $(n-1)$ (both are unique up to scalar multiplications).

Let $\pi:\T^1(\bH^n)\to \bH^n$ be the basepoint projection. For $u\in \T^1(\bH^n)$, we denote by $u^{\pm}\in \partial \bH^n$ the forward and the backward endpoints of the geodesic determined by
$u$. Fix $o\in \bH^n$ so that $K$ fixes $o$. The map
\[
u \mapsto (u^+, u^-, s=\beta_{u^-} (o,\pi(u)))
\]
is a homeomorphism  between $\T^1(\bH^n)$ and $( \partial\bH^n\times \partial\bH^n- \{(\xi,\xi):\xi\in \partial\bH^n\})  \times \R.$
  In these coordinates, the BMS measure $\bms=\m^{\rm{BMS}}$, 
 the Haar measure $\m^{\rm{Haar}}$, and the Burger-Roblin measure $\m^{\rm{BR}}$ on $\T^1(\bH^n)$ are given by
\begin{enumerate}
 \item 
$
d\bms(u)= e^{\delta \beta_{u^+}(o,
\pi(u))}\;
 e^{\delta \beta_{u^-}(o, \pi(u)) }\; d\nu_o(u^+) d\nu_o(u^-) ds.
$

\item  $ d\m^{\rm{Haar}}(u)= e^{(n-1)\beta_{u^+}(o,
\pi(u))}\;
 e^{(n-1) \beta_{u^-}(o, \pi(u)) }\; dm_o(u^+) dm_o(u^-) ds.
$

\item
$ d\m^{\rm{BR}}(u)= e^{(n-1) \beta_{u^+}(o,
\pi(u))}\;
 e^{\delta \beta_{u^-}(o, \pi(u)) }\; dm_o(u^+) d\nu_o(u^-) ds.
 $

\end{enumerate}

These measures are all left $\G$-invariant, and hence
descend to corresponding measures on $\T^1(\cM)$ .
Using $\T^1(\bH^n)=G/M$, we can lift the above measures to right $M$-invariant measures on $\Gamma\ba
G$, which we still denote by $\bms$, $\m^{\Haar}$ and $\m^{\BR}$ by abuse of notation. The measure $\bms$ is  finite and ergodic with respect to the geodesic flow  \cite{Sul82}. 
We will normalize the Patterson-Sullivan density $\{\nu_x\}$ so that $\m(\T^1(\cM))=\bms(\G\bk G)=1$.

Let $N=N^+$ and $N^-$ denote the expanding and the contracting horospherical subgroups respectively, i.e.,
$$N^{\pm}=\{g\in G: a_s ga_{-s}\to e\text{ as $s\to \pm \infty$}\}.$$

Note that $$\Omega:=\op{supp}(\m)=\{[g]\in \G\ba G: g^+, g^-\in\Lambda(\G)\},$$
where  $g^{\pm}:=[gM]^{\pm} \in \partial\bH^n$.

The BMS measure $\bms$ has a natural foliation corresponding to the decomposition $PN=G$ (modulo a Zariski closed subset)
with $P=N^-AM$.
Explicitly, for any $g\in G$, we define the PS-measure and the Lebesgue measure on the coset $gN$, by
\begin{equation} \label{e: muPS} 
d\tilde \mu_{gN}^{\PS}(gn)= e^{\delta  \beta_{(gn)^+}(o,gn)} d\nu_o (gn)^+,
\end{equation}
and 
\begin{equation}\label{e:muLeb} 
d\tilde \mu_{gN}^{\Leb}(gn)= e^{(n-1) \beta_{(gn)^+}(o,gn)} dm_o(gn)^+,
\end{equation}
respectively. We  also define the measure $\tilde\nu_{gP}$ on the coset $gP$ by
\begin{equation}\label{e:nu}
d\tilde \nu_{gP}(gp)=e^{\delta t} d\nu_o(gp)^-dt\end{equation}
 for $t=\beta_{(gp)^{-}}(o, gp)$.
Using the decomposition $G=gPN$ and noting that $(gpn)^-=(gp)^-$, we have that for any $\Psi\in C_c(G)$, 
\begin{equation}\label{e:mPsi} \m(\Psi)=\int_{gP}\int_{N}\Psi(gpn)d\tilde \mu^{\PS}_{gpN}(gpn) d\nu_{gP}(gp) .\end{equation}

Finally, for $x=[g]\in \Gamma\ba G$ and $\e>0$ smaller than the injectivity radius at $x$, we denote by $d\mu_{xN_\e}^{\PS}$ and $d\nu_{xP_\e}$  the  measures 
induced by $d\tilde \mu_{gN}^{\PS}$ and $d\tilde\nu_{gP}$ on  $xN_\e$ and $xP_\e$ respectively.

\subsection{Cusp decomposition}\label{s:cusps}
Let $X_0$ be the pre-image of the convex core of $\cM$
under the base point  projection map $\pi: \Gamma\ba G\to \Gamma\ba G/K=\cM$ and let $X$ be the unit neighborhood of $X_0$. Then $\Omega\subseteq X_0\subseteq X$  and since $\cM$ is geometrically finite, $X$ has finite Haar-measure. When $\cM$ is convex cocompact, $X$ is compact, and otherwise it can be decomposed into a compact part and finitely many cusp neighborhoods, as we describe below.

Let $\Lambda_p\subset \Lambda$ denote the set of parabolic fixed points (i.e. points fixed by some parabolic element of $\G$). Since $\G$ is geometrically finite, $\Lambda_p$ consists of finitely many $\G$-orbits represented by $\{\xi_1,\ldots,\xi_k\}$ which are called
cusps of $\cM$. A cuspidal neighborhood of $\xi_i\in \Lambda_p$ is a set of the form 
\begin{equation}\label{hii} \fh_i=\pi^{-1} (\G\bk \G H_{\xi_i})\end{equation}
 where $H_{\xi}\subseteq \bH^n$ is some fixed horoball tangent to $\xi$ such that
 $\gamma H_{\xi}\cap H_{\xi}\ne \emptyset $ if and only if $\gamma$ fixes $\xi$. For each $i$,
 the stabilizer $\op{Stab}_\Gamma (\xi_i)$ is a free abelian subgroup and we denote its rank by $\kappa_i$.
 We set $\kappa_{\rm max}:=\max \kappa_i$ and $\kappa_{\rm min}:=\min \kappa_i$. Note that 
 $2\delta > k_{\rm max}$(see \cite[Lem. 3.5]{CI99}).

For $x\in \G\bk G$, we denote by $r_x$ the injectivity radius at $x$.
For all sufficiently small $\epsilon>0$, let 
$X(\epsilon)=\{x\in X:r_x<\epsilon\},$
so that 
 $$Y(\epsilon):=X\setminus X(\epsilon)$$ is compact, and the family $X(\epsilon)$ with  $\epsilon<\e_0$ forms a shrinking family of cusp neighborhoods.

More explicitly, we show in section \ref{s:cusp} that for all sufficiently small $\e>0$,
\begin{equation}\label{e:inj}
X(\e)\cap \fh_i  \asymp X\cap \fh_{i, \log(\e^{-1})},
\end{equation}
and using the measure estimate $\m(\fh_{i, \log(\e^{-1})}) \asymp \e^{2\delta-\kappa_i}$ (see Proposition \ref{p:BMScusp}),
we get that
\begin{equation}\label{e:mXepsilon}
\m(X(\e))\asymp  \e^{2\delta-\kappa_{\rm max}}.\end{equation}

 \subsection{Sobolev norms}\label{sob}
 The mixing rate of the geodesic flow depends on the smoothness of the test functions which can be captured by appropriate Sobolev norms we now define.
Given a fixed basis of $\op{Lie} G$,
$l \in \N$, and $1\le p\le \infty$,
the Sobolev norm $\cS_{p,l}(\Psi)$ of $\Psi\in C^\infty(\G\bk G)$ is defined by
\begin{equation} \label{si}\mathcal \cS_{p, l}(\Psi)=\sum \| D (\Psi)\|^{\Haar}_{p} \end{equation}
where the sum is taken over all monomials  $D$ of order at most $l$ in the basis elements,
and $\|\Psi\|^{\Haar}_p$ denotes the $L^p(\G\bk G, \m^{\Haar})$-norm of $\Psi$. While this norm depends on the choice of basis, changing the basis will only change the norm by some bounded factor. 

We will mostly use the norms $\cS_{\infty,l}$, which we will denote by $\cS_l$ to simplify notation.
Since $\supp(\m)\subset X$, it is sufficient for our purpose to consider functions supported on $X$, and since 
 $X$ has finite Haar measure we can, and will use the bound
$$\cS_{p,l}(\Psi)\leq \cS_{l}(\Psi)\m^{\rm{Haar}}(X)^{1/p}\ll \cS_{l}(\Psi),$$
where the implied constant is independent of $\Psi\in C^\infty(X)$.

\section{Decay of matrix coefficients}\label{s:mix}
A crucial ingredient in our proof is the exponential mixing of the geodesic flow with respect to the BMS-measure. 
We use the inner product notation:
$$\la a_t \Psi, \Phi \ra =\int_{\G\bk G} \Psi(xa_t) \Phi(x) \; d\m (x) .$$

{ By the remarks following Theorem \ref{mix}, this following theorem is the only missing part of it, given the works \cite{Wi} and \cite{LiPan20}. }

\begin{thm}\label{t:mmix} { Suppose that $\delta> \max\{\tfrac{n-1}{2},n-2\}$ (resp. $\delta>\tfrac{n-1}{2}$).}
Then  there exist an explicit $\eta_0>0$ (depending only on the spectral gap of $L^2(\mathcal M)$) and $l\in \N$,
such that 
for any bounded $\Psi, \Phi\in C^\infty(X)$ (resp. $\Psi, \Phi\in C^\infty(X)^M$) 
$$\la a_t \Psi, \Phi \ra
 = {\m(\Psi)\cdot \m(\Phi)} +O( e^{-\eta_0 t}   \cS_l(\Psi) \cS_l(\Phi)). $$
\end{thm}

In the rest of this section, we assume 
$$\delta>(n-1)/2 .$$

Theorem \ref{t:mmix} with an explicit $\eta_0$ depending only on the spectral gap of $L^2(\cM)$
 is then proved in  (\cite[Theorem 6.16]{MO15}, \cite{EO}) under the assumption that the test functions are compactly supported. In order to complete the proof of the theorem we need to remove the assumption on the support of the test functions.

To do this, we will approximate $\Psi$ as the sum $\Psi_{ \e} +(\Psi -\Psi_{\e})$
where $\Psi_{\e}$ is a smooth function supported on $Y(\e )$, and similarly for $\Phi$.
In view of \eqref{e:mXepsilon}, the main term will be reduced to
$\la a_t \Psi_{\e}, \Phi_{ \e}\ra$, for which the result follows from \cite[Theorem 6.16]{MO15}. However, since the dependence on the supports of 
$\Psi_{\e}$ and $ \Phi_{\e}$ was not made
explicit in terms of $\e$ in \cite{MO15}, we need to redo their arguments while keeping track of
the dependence on $\e$ as well as on all implied constants along the proof.

\subsection{Control of BR measures}
Since $\m^{\BR}(\G\bk G)=\infty$ when $\G <G$ is not a lattice, and some of the implied constants in \cite[Thm. 6.16]{MO15} depend on $\m^{\BR}(\supp(\Psi))$, we need the following result to control the dependence on these measures. 

\begin{lem}\label{br2}
Assume that $\delta>\tfrac{n-1}{2}$. 
Then there exists $c>0$ such that for any $K$-invariant subset $Y\subset\G\bk G$ with $\m^{\Haar}(Y)<\infty$,
we have $$\m^{\BR}(Y)\le c \cdot \sqrt{\m^{\Haar}(Y)}.$$
\end{lem}
\begin{proof}
Recall that by \cite{Sul79} and \cite{LP82}, there exists a positive eigenfunction
 $\phi_0\in C^\infty(\Gamma\ba G)^K$ for the Laplace operator such that
$$-\Delta \phi_0= \delta(n-1-\delta)\phi_0.$$
Under the assumption $\delta>\tfrac{n-1}{2}$, we have $\|\phi_0\|^{\Haar}_2<\infty .$
If $\Psi$ denotes  the indicator function of $Y$, then $\Psi$ is $K$-invariant and hence by  \cite[Lem. 6.7]{KO11}
$$\m^{\BR}(\Psi)=\int_{X} \Psi(x) \phi_0(x) d\m^{\Haar}(x), $$
and in particular 
$\m^{\BR}(Y)\leq  \|\phi_0\|^{\Haar}_2\sqrt{\m^{\Haar}(Y)}$,
as claimed. 
\end{proof} 

Since $X$ is $K$-invariant with $\m^{\Haar}(X)<\infty$, the following follows from Lemma \ref{br2}:
\begin{Cor} If $\delta>\frac{(n-1)}{2}$,
then  $ \m^{\BR}(X)<\infty .$
\end{Cor}

\subsection{Test function supported on small balls}
For a subset $S\subseteq G$ and $\e>0$, $S_\e$ denotes the $\e$-neighborhood of $e$ in $S$, that is, 
$S_{\e}=\{g\in S: d(g, e)\le \e\}$.
Set $B_\e:=P_{\e}N_\e$; and note that $G_\e\asymp B_\e$ for all sufficiently small $\e>0$.
In this subsection, we will prove the following.

\begin{prop}\label{p:effe}  Suppose $\delta> \max\{\tfrac{n-1}{2},n-2\}$ (resp. $\delta>\tfrac{n-1}{2}$). There exist  $l\in \N$ depending only on $\dim(G)$ and $\eta>0$ (depending only on the spectral gap of $\Gamma$) such that for any $\e\in(0,1)$ small and  any $x\in Y(\e)\cap \Omega$, 
 for all  $\Phi, \Psi\in C^\infty(xB_{\e})$ (resp. $\Phi, \Psi\in C^\infty(xB_{\e} M)^M$),  we have that 
\begin{align} \label{effe} &
\la a_t \Psi, \Phi \ra = \bms(\Psi) \bms(\Phi)+ O(e^{-\eta t}  \cS_{l}(\Psi) \cS_{l}(\Phi))
\notag \end{align}
where the implied constant is absolute.
\end{prop}
\begin{proof} Fix $\Phi, \Psi\in C^\infty(xB_{\e})$.
In the case when $\frac{n-1}2< \delta\le n-2$, we assume that $\Phi, \Psi\in C^\infty(xB_{\e}M)$ are $M$-invariant.
We have
\begin{align*} \la a_t \Psi, \Phi \ra
&=\int_{xp\in xP_{\e}}\int_{xpN_{\e}} \Psi(xpna_t) \Phi(xpn) d\mu_{xpN}^{\PS} (xpn) d\nu_{xP}(xp).
\end{align*}
Now, for fixed $p\in P_\e$, letting $\phi=\Phi |_{xpN_{\e}}\in C^\infty_c(xpN_\e)$, we estimate the inner integral 
\begin{equation}
\int_{xpN_{\e}} \Psi(xpna_t) \Phi(xpn) d\mu_{xpN}^{\PS} (xpn)=\int_{xpN_{\e}} \Psi(xpna_t) \phi(xpn) d\mu_{xpN}^{\PS} (xpn)
\end{equation}
as follows.

Fix a small $0<\e_0<\e^2$ and consider the functions $\Psi_{\e_0}^{\pm}$ on $\G\bk G$ defined by 
$$\Psi_{\e_0}^{+} (y)= \sup_{g\in G_{\e_0}} \Psi(yg),\quad
\Psi_{\e_0}^{-} (y)= \inf_{g\in G_{\e_0}} \Psi(yg)$$ 
and let 
$$\psi_{\e_0}^{\pm}(xp)=\int_{xpN}\Psi_{\e_0}^{\pm} (xpn) d\mu_{xpN}^{\PS}(xpn).$$
We then have that $$\nu_{xP}(\psi_{\e_0}^{\pm})=\m(\Psi_{\e_0}^{\pm})
\quad \text{ and} \quad \int_{xP_{\e}}\mu_{xpN}^{\PS}(\phi)d\nu_{xP}(xp)=\m(\Phi) .$$
Moroever, since $\Psi(x)=\Psi_{\e_0}^\pm(x)+O(\e_0\cS_{\infty,1}(\Psi))$, we get that 
$$\bms(\Psi_{\e_0}^\pm)=\bms(\Psi)+O(\e_0\cS_{\infty,1}(\Psi)),$$
where we used that $\bms(X)<\infty$.
We will also use the notation  
$$\phi^+_{\e_1}(y):=\sup_{n\in N_{\e_1}} \phi(yn),$$  
and similarly get that 
$\mu_{yN}^{\PS}(\phi_{\e_1}^{+})=\mu_{yN}^{\PS}(\phi)+O(\e_1 \cS_{\infty,1}(\phi))$.

Now by (\cite[Lem. 6.2]{MO15}, \cite{EO}),  there exists some absolute constant $c>0$, such that 
the integral \begin{equation}\label{e:innerint}
\int_{xpN_{\e}} \Psi(xpna_t) \phi(xpn) d\mu_{xpN}^{\PS} (xpn)
\end{equation}
 is bounded from above and below, respectively, by 
$$
 (1\pm c\epsilon_0) e^{-\delta t} \sum_{p\in \cP_x(t)} \psi_{c\e_0}^{\pm}(xp) \phi^\pm_{ce^{-t}\e_0} (xpa_{-t}),$$
where $\cP_x(t)$ is the finite set defined by 
 $$\cP_x(t)=\{p\in P_{\e}: xpN_{\e} a_t\cap xpN_{\e}\ne \emptyset\}.$$
Moreover, by the proof of \cite[ Thm. 6.7]{MO15}, there are positive constants  $\eta>0$ (depending only on the spectral gap of $\Gamma$) and  $\alpha>0$ such that
\begin{align*} 
e^{-\delta t} \sum_{p\in \cP_x(t)} \psi_{\e_0}^{\pm}(xp) \phi^+_{e^{-t}\e_0} (xpa_{-t}) &= \nu_{xP}(\psi_{\e_0}^{\pm}) \mu_{xpN}^{\PS}(\phi^{\pm}_{e^{-t}\e_0})\\
&+
O( e^{-\eta t} + \e_0^{\alpha}) A_\Psi^{\BR} A_\phi^{\PS}+ O(e^{-\eta t} \cS_{2,l} (\Psi) \cS_{2,l}(\phi))
\end{align*}
where  
$$A_\Psi^{\BR}:=\cS_{\infty,1}(\Psi) \m^{\BR} (\text{supp}(\Psi))\ll \cS_{\infty,1}(\Psi),$$
(by Lemma \ref{br2}), and
 $$A_\phi^{\PS}:=\cS_{\infty,1}(\phi) \mu_{xpN}^{\PS} (\text{supp}(\phi))\leq \cS_{\infty,1}(\phi)\mu_{xpN}^{\PS} (xpN_\e).$$
 Notice that the injectivity radii of the supports of $\phi$ and $\psi$ are at least $\e$ and since we chose $\e_0\ll\e^2$ much smaller, all the implied constants are absolute and independent of $\e$ and $\e_0$.

Combining these results and estimating 
$$\nu_{xP}(\psi_{\e_0}^\pm)=\m(\Psi_{\e_0}^\pm)=\bms(\Psi)+O(\epsilon_0 \cS_{\infty,1}(\Psi)),$$ 
and 
$$\mu_{xpN}^{\PS}(\phi^{\pm}_{e^{-t}\e_0})=\mu_{xpN}^{\PS}(\phi)+O(\epsilon_0 \cS_{\infty,1}(\phi)),$$
we get that 
\begin{multline*}
\int_{xpN_{\e}} \Psi(xpna_t) \phi(xpn) d\mu_{xpN}^{\PS} (xpn)= \bms(\Psi)\mu_{xpN}^{\PS}(\phi)(1+O(\epsilon_0))\\
+O(\epsilon_0 \cS_{\infty,1}(\Psi)\cS_{\infty,1}(\Psi))
+O( e^{-\eta t} + \e_0^{\alpha}) \cS_{\infty,1}(\Psi)\cS_{\infty,1}(\phi)\mu_{xpN}^{\PS} (xpN_\e)\\
+ O(e^{-\eta t} \cS_{2,l} (\Psi) \cS_{2,l}(\phi)).
\end{multline*}

Since all implied constants are independent of $\e_0$, taking the limit as $\e_0\to 0$ 
gives 
\begin{eqnarray*}
&&\int_{xpN_{\e}} \Psi(xpna_t) \phi(xpn) d\mu_{xpN}^{\PS} (xpn)=\bms(\Psi)\mu_{xpN}^{\PS}(\phi)\\
&&+O( e^{-\eta t}\cS_{l}(\Psi)\cS_{l}(\Phi)\mu_{xpN}^{\PS} (xpN_\e)))+ O(e^{-\eta t} \cS_{2,l} (\Psi) \cS_{2,l}(\phi))
\end{eqnarray*}
where we used that $\cS_{l}(\phi)\leq \cS_{l}(\Phi)$.

Now, integrating over $xP_\e$, and noting that
$\int_{xP_\e}\mu_{xpN}^{\PS}(\phi)d\nu_{xP}(xp)=\bms(\Phi)$, 
the main term is indeed $\bms(\Psi)\bms(\Phi)$. Next, since 
$$\int_{xP_\e}\mu_{xpN}^{\PS} (xpN_\e))d\nu_{xP}(xp)=\int_{xP_\e}\int_{xpN_\e}d\mu_{xpN}^{\PS} d\nu_{xP}(xp)=\bms(B_\e)\leq 1,$$
the integral of the first remainder term is bounded by $O(e^{-\eta t}\cS_{l}(\Psi)\cS_{l}(\Phi))$. For the second remainder term, we bound 
$\cS_{2,l}(\phi)\leq \cS_{l}(\Phi)\sqrt{\mu_{xpN}^{\Leb}(xpN_\e)}$ to get that 

\begin{eqnarray*}
\int_{xP_\e}\cS_{2,l}(\Phi_{|_{xpN_\e}})d\nu_{xP}(xp)&\ll& 
\cS_{l}(\Phi)(\mu_{xpN}^{\Leb}(xpN_\e))^{-1/2}\int_{xP_\e}\int_{xpN_\e}d\mu_{xpN}^{\Leb}(xpn)d\nu_{xP}(xp)\\
&=&\cS_{l}(\Phi)(\mu_{xpN}^{\Leb}(xpN_\e))^{-1/2}\m^{\BR}(xP_\e N_\e).
\end{eqnarray*}
We now use Lemma \ref{br2} to bound
$$\m^{\BR}(xP_\e N_\e)\leq \m^{\BR}(xP_\e N_\e K)\ll \sqrt{\m^{\Haar}(xP_\e N_\e K)},$$
and since there is a uniform constant $c>0$ such that $P_\e N_\e K\subseteq P_{c\e}K$, noting that $P_\e K=N^{-}_\e A_{\e} K$,  we can bound 
$$\m^{\Haar}(xP_\e N_\e K)\ll \m^{\Haar}(xP_{c\e}K)\ll \mu_{xpN}^{\Leb}(xpN_\e)$$ to get that 
\begin{eqnarray*}
\int_{xP_\e}\cS_{2,l}(\Phi_{|_{xpN_\e}})d\nu_{xP}(xp)\ll \cS_{l}(\Phi).
\end{eqnarray*}

Combining the two remainder terms, and bounding all norms by $\cS_l(\Psi)\cS_l(\Phi)$ we get that 
\begin{eqnarray*} 
\la a_t \Psi, \Phi \ra =\bms(\Psi)\bms(\Phi)+O(e^{-\eta t}\cS_l(\Psi)\cS_l(\Phi))
\end{eqnarray*}
where the implied constant is absolute.
\end{proof}

\subsection{General test functions}
We now use a partition of unity to reduce the case of a general test function to the case of functions with small support.

For $\epsilon\in(0,1)$ sufficiently small, let $Q_\e$ be a maximal family of points in $X\cap Y_\e$ such that the sets $yB_{\e^3}$, $ y\in Q_\epsilon$,
are disjoint and meet $Y_{2\e}$, and let $Q_\e':=\{y\in Q_\e: yB_{\e^2}\cap Y_{4\e}\ne \emptyset\}$. 
Note that the collection $\{yB_{\e^2}:y\in Q_\e\}$ covers $X\cap Y_{2\e}$ and
 that $\{yB_{\e^3}B_{\e^3}  : y\in Q_\e'\}$ covers $X\cap Y_{4\e}$. Since $\m^{\Haar} (X)<\infty$,
 we have $\# Q_\e= O(\e^{-3\dim(G)})$.
 
Fix a non-negative function $ \beta_\e\in C^\infty(B_\e)$ taking values in $[0,1]$  which is $1$ on $B_{\e^3}B_{\e^3}$ and $0$ outside $B_{\e^2}$ (note that $B_{\e^3}B_{\e^3}\subseteq B_{2\e^3}\subset B_{\e^2}$). We can choose $\beta_{\e}$ so that $\cS_{l}(\beta_{\e})\ll \epsilon^{-3l}$.
For each $y\in Q_\e$, define a function  $\beta_{y,\epsilon}(yb):=\beta_\epsilon(b)$ on $yB_\epsilon$.
\begin{lem}
For any $x\in\bigcup_{y\in Q_{\e}'} yB_{\e^2}$, we have  
$$\sum_{z\in Q_\e}\beta_{z,\epsilon}(x)\geq 1.$$
\end{lem}
\begin{proof}
Let $y\in Q'_\e$. If $x\in yB_{\e^3}B_{\e^3}$, then  $\beta_{y,\e}(x)=1$ and hence $\sum_{z\in Q_\e}\beta_{z,\epsilon}(x)\geq 1$.
Now suppose that $x\in yB_{\e^2}\setminus yB_{\e^3}B_{\e^3}$, in which case
$xB_{\e^3}\cap yB_{\e^3}=\emptyset$. Since $y\in Y_{3\e}$ and $x\in yB_{\e^2}$ we have that $x\in Y_{2\e}\cap X$. By the maximality of $Q_\e$,
 there exists $z\in Q_\e$ such that $xB_{\e^3}\cap zB_{\e^3}\neq \emptyset$. This implies that $x\in zB_{\e^3}B_{\e^3}$ and hence $\beta_{z,\e}(x)\geq 1$.
\end{proof}
Now consider
the normalized function (supported on $yB_\e$) given by
$$\alpha_{y, \e}:=\frac{\beta_{y, \e}}{\sum_{z\in Q_\e} \beta_{z, \e}} .$$
\begin{lem}\label{l:alphaybound}
For any $y\in Q_\e'$, we have
 $\cS_{l}(\alpha_{y, \e})\ll \e^{-p},$
 where  $p$ and the implied constant depend only on $l$ and $\dim(G)$. 
 \end{lem}
 \begin{proof}
 Let $s_\e(x)=\sum_{z\in Q_\e} \beta_{z, \e}(x)$ so that $\alpha_{y,\e}(x)=\frac{\beta_{y,\e}(x)}{s_\e(x)}$.
Since $\alpha_{y,\e}$ is supported on $yB_{\e^2}$, we only need to bound its derivatives there in which case we have that $s_\e(x)=\sum_{z\in Q_\e} \beta_{z, \e}(x)\geq 1$. Taking derivatives of the quotient  $\alpha_{y,\e}=\frac{\beta_{y,\e}}{s_\e}$ and using the bound $s_\e(x)\geq 1$ together with the bound
 $\cS_{\infty, l}(s_\e)\ll \e^{-3l}\# Q_{\e}\ll \e^{-3(l+\dim(G))}$ proves the lemma. \end{proof}
\begin{lem}
The function $\tau_\e:=\sum_{y\in Q_\e'} \alpha_{y, \e}$ belongs to $C^\infty(X)$ and satisfies that
$0\le \tau_\e \le 1$, $\tau_\e=1$ on $X\cap Y_{4\e}$, and $\tau_\e=0$ outside $Y_{\e}$.
\end{lem}
\begin{proof}
Since $\tau_\e=\frac{\sum_{y\in Q_\e'}\beta_{y,\e}}{\sum_{y\in Q_\e}\beta_{y,\e}}$, it is clear that $0\leq \tau_\e\leq 1$. 
Note that if $y\in Q_\e \setminus Q_{\e}'$, then $yB_{\e^2}\cap Y_{4\e}=\emptyset$. Hence 
 if $x\in X\cap Y_{4\e}$ satisfies
$x\not\in yB_{\e^2}$, then $\beta_{y,\e}(x)=0$. This shows that $\sum_{y\in Q_\e}\beta_{y,\e}(x)=\sum_{y\in Q_\e'}\beta_{y,\e}(x)$. Moroever, since $X\cap Y_{4\e}$ is covered by $\{yB_{\e^3}B_{\e^3}:y\in Q_\e'\}$, we have that $\sum_{y\in Q_\e'}\beta_{y,\e}(x)\neq 0$ on $X\cap Y_{4\e}$ and hence indeed $\tau_\e=1$ there.  Next, since for any $y\in Q_{\e}'$, we have that $yB_{\e^2}\subseteq Y_{\e}$; so $\tau_{\e}(x)=0$ outside of $Y_{\e}$. Finally we can bound
$\cS_{l}(\tau_{\e})\leq \sum_{y\in Q_{\e}'}\cS_{l}\alpha_{y,\e}\ll \e^{-p+3\dim(G)}.$ \end{proof}

\begin{proof}[Proof of Theorem \ref{t:mmix}]
Suppose first that $\delta>\max\{\tfrac{n-1}{2},n-2\}$. 
Now, for given $\Psi ,\Phi\in C^\infty(X)$, 
consider 
$$\Psi_\e:=\Psi\cdot \tau_\e=\sum_{y\in Q_\e'} \Psi\cdot \alpha_{y, \e}\;\;\;\text{and}\;\; \;\Phi_\e:=\Phi\cdot \tau_\e= \sum_{y\in Q_\e'} \Phi\cdot \alpha_{y, \e}.$$
Note that $\cS_{l}(\Psi \cdot \alpha_{y,\e})\ll \cS_{l}(\alpha_{y,\e})\cS_{l}(\Psi)\ll \e^{-p}\cS_{l}(\Psi),$ with $p$ as in Lemma \ref{l:alphaybound}. Now
applying Proposition  \ref{p:effe} to each $\Psi \cdot \alpha_{y, \e}$ and $\Phi \cdot \alpha_{y', \e}$ for  $y,y'\in Q_\e'$, and recalling that $\#Q_\e=O(\e^{-3\dim G})$,
we get that 
\begin{eqnarray}\label{mmm} 
\la a_t \Psi_\e, \Phi_\e \ra &=& \bms(\Psi_\e) \bms(\Phi_\e)+O(\e^{-p_0}e^{-\eta t}\cS_{l}(\Psi)\cS_{l}(\Phi))
\end{eqnarray}
with $p_0=2p+6\dim(G)$.

It follows from  \eqref{e:mXepsilon} that  for $\delta_0:=2\delta -\kappa_{\rm max}>0$,
$$\m(\Psi-\Psi_\e) \le \|\Psi\|_\infty \m(X_{4\e}) \ll  \e^{\delta_0} \|\Psi\|_\infty,$$
and similarly $ \m(\Phi-\Phi_\e) \ll  \e^{\delta_0} \|\Phi\|_\infty $. 
Hence
$$|\la a_t \Psi, \Phi \ra-\la a_t \Psi_\e, \Phi_\e \ra|\ll \e^{\delta_0} \|\Psi\|_\infty\|\Phi\|_\infty.$$
We then deduce
\begin{eqnarray*}
\la a_t \Psi, \Phi\ra &=& \bms(\Psi) \bms(\Phi)+O(\epsilon^{\delta_0}\|\Psi\|_\infty\|\Phi\|_\infty)+O(\e^{-p_0}e^{-\eta t}\cS_{l}(\Psi)\cS_{l}(\Phi)).
\end{eqnarray*}
Taking $\epsilon=e^{-\frac{\eta t}{\delta_0+p_0}} $ and recalling that $\cS_{\infty,0}\ll \cS_l$, we get that 
\begin{eqnarray*}
\la a_t \Psi, \Phi\ra &=& \bms(\Psi) \bms(\Phi)+O(e^{-\eta_0 t}\cS_l(\Psi)\cS_l(\Phi))
\end{eqnarray*}
with $\eta_0=\frac{\eta \delta_0}{\delta_0+p_1}$. This concluds the proof when $\delta> \max\{\tfrac{n-1}{2},n-2\}$.

Finally, for $n>3$, if $\frac{n-1}2< \delta\le n-2$, and $\Psi$ and $\Phi$ are $M$-invariant,
we can replace $\alpha_{y, \e}$ with an $M$-invariant function $\alpha^M_{y,\e}(x)=\int_M \alpha_{y,\e}(xm) dm$ and run the same argument
to get \eqref{mmm}. Then the rest of the proof is identical.
\end{proof}

\section{Shrinking target problems}\label{s:shrinking}
We now use the results on the exponential decay of matrix coefficients to prove an effective mean ergodic theorem and apply it to various shrinking target problems. 
As before, we assume that $\Gamma$ is a geometrically finite, Zariski dense subgroup of $G=\op{SO}(n,1)^\circ$. 
{ For $n\geq 5$, in the case where $\Gamma$ has a cusp and $\delta\leq n-2$, all functions and shrinking targets
on $X$ we consider below are  assumed to be $M$-invariant so that Theorem \ref{mix} applies to them.
All functions  below are also assumed to be real-valued functions.}

\begin{rem}
While we state our results for the geodesic flow on geometrically finite hyperbolic manifolds, we note that the results in this section are quite general and hold for any dynamical system on a measure space $(X,\bms)$ for which one has exponential decay of correlation in the sense of  Theorem \ref{t:mmix}
\end{rem}

\subsection{Effective mean ergodic theorem}\label{s:EMET}
Fix $\ell$ as given in Theorem \ref{t:mmix}.  For notational convenience, we introduce the norm
$$\cS^*(\Psi):=\frac{\cS_l(\Psi)}{\|\Psi\|} \quad\text{for any non-zero $\psi\in C^\infty(X)\cap L^2(X,\m)$, }$$
where $\|\Psi\|$ denotes the $L^2$-norm of $\Psi$.
In the entire section, we will take $\lambda$ to be either the Lebesgue measure on $\R$ (when considering a continuous time flow) or the counting measure on $\Z$ (for a discrete time flow). For $T\geq 1$, consider the averaging operator $\lambda_T$ on $L^2(X,\bms)$ given by 
$$\lambda_T(\Psi)(x)=\frac{1}{T}\int_0^T\Psi(xa_t)d\lambda(t).$$

\begin{thm} \label{t:MET}\label{mean}
 For any non-zero $\Psi\in C^\infty(X)$, and for all $T\gg 1$,
$$\|\lambda_T(\Psi)-\bms(\Psi)\|^2 \ll \frac{(1+\log(\cS^*(\Psi) ))\cdot \|\Psi\|^2}{T}.$$
\end{thm}
\begin{proof}
Since we have
$\|\lambda_T(\Psi)-\bms(\Psi)\|^2=\|\lambda_T(\Psi)\|^2-\bms(\Psi)^2,$
it is enough to estimate $\|\lambda_T(\Psi)\|^2$.
Now, expand
\begin{eqnarray*}
\|\lambda_T\Psi\|^2&=&\frac{1}{T^2}\int_0^T\int_0^T \int_{X}\Psi(xa_{t_1-t_2}) {\Psi(x)}d\bms(x)d\lambda(t_1)d\lambda(t_2)\\
&=&\frac{1}{T^2}\int_{-T}^T  \int_{X}\Psi(xa_{t}) {\Psi(x)}d\bms(x)(T-|t|)d\lambda(t)\\
\end{eqnarray*}
where we used that  $\lambda$ is translation invariant and $\lambda([0,T)\cap [t,t+T))=T-|t|$ (where in the discrete case we may and will assume that $t$ and $T$ are integers).

Now fix a large parameter $M$ to be determined later. For $|t|\geq M$ large we use Theorem \ref{mix} to get that 
$$ \int_{X}\Psi(xa_{t}) {\Psi(x)}d\bms(x)=\bms(\Psi)^2+O(\cS(\Psi)^2e^{-\eta_0 |t|}),$$
for some $\eta_0\in (0,1)$. On the other hand,
for $|t|<M$ small, we bound $\<a_t\Psi,\Psi\> \leq \|\Psi\|^2$, to get that
\begin{eqnarray*}
\|\lambda_T\Psi\|^2&=&\bms(\Psi)^2+O(\|\Psi\|^2\tfrac{M}{T})+O(\tfrac{\cS(\Psi)^2e^{-\eta_0 M}}{T}),\\
\end{eqnarray*}
using $\bms(\Psi)\leq \|\Psi\|$.  Using these estimates, we get that
$$
\|\lambda_T(\Psi)-\bms(\Psi)\|^2\ll \frac{M\|\Psi\|^2+\cS(\Psi)^2e^{-\eta_0 M}}{T}.
$$
It remains to set $M={2\log(\cS^*(\psi))}{\eta_0}^{-1}$ to finish the proof.
\end{proof}

Following \cite{Kel17}, for a non-negative function $\Psi$ on $X$, we define 
\begin{align} \label{r:Ctf}
&\cC_{T,\Psi}:=\{x\in X: |\lambda_T(\Psi)(x)-\bms(\Psi)|\geq \tfrac{\bms(\Psi)}{2}\};\\
 &\cC_{T,\Psi}^o:=\{x\in X: \lambda_T(\Psi)(x)=0\}.
\end{align}
Note that  
$\cC_{T,\Psi}^o\subseteq \cC_{T,\Psi}$. 

As a direct consequence of the effective mean ergodic theorem, we get the  following bounds:
\begin{prop}\label{p:mCf}
For a non-negative $\Psi\in C^\infty(X)$  and $T\geq 1$, we have 
$$\bms(\cC_{T,\Psi})\ll \frac{\log({\cS^*(\Psi)}) \|\Psi\|^2}{T \cdot \bms(\Psi)^2}.$$
\end{prop}
\begin{proof}
On one hand, 
$$\|\lambda_{T}(\Psi)-\bms(\Psi)\|^2\geq \int_{\cC_{T,\Psi}} |\lambda_{T}(\Psi)(x)-\bms(\Psi)|^2d\bms\geq \frac{(\bms(\Psi))^2\bms(\cC_{T,\Psi})}{4}.$$
On the other hand, by Theorem \ref{t:MET},
$$\|\lambda_{T}(\Psi)-\bms(\Psi)\|^2\ll \frac{\log({\cS^*(\Psi)}) \|\Psi\|^2}{T}.$$
Putting these two together gives the result. \end{proof}

Having control on the measures of these subsets has immediate consequences to several shrinking target problems. 
Indeed, a simple adaptation of \cite[Lemmas 13 and 14]{Kel17} gives the following result.
\begin{lem}\label{l:Cfk}
Let $\{\Psi_t\}_{t\geq 1}\subseteq L^2(X,\bms)$ be a decreasing family of bounded non-negative functions.
\begin{enumerate}
\item If  $\sum_j \bms(\cC^o_{t_{j-1}, \Psi_{t_j}})<\infty$ for some subsequence $t_j\to \infty$,
then for $\bms$-a.e. $x\in X$, 
 $$\lambda_t \Psi_t(x)\neq 0\quad\text{ for all $t\gg_x 1$. }$$
\item If  there exists $C>1$ such that $\bms(\Psi_{2^j})\le C\cdot  \bms(\Psi_{2^{j+1}})$ for all $j\gg 1$ and 
$\sum_j \bms(\cC_{2^{j-1},\Psi_{2^j}})<\infty,$
then for $\bms$-a.e. $x\in X$, $$\lambda_t(\Psi_t)(x)\geq \frac{\bms(\Psi_t)}{4C}
\quad \text{for all $t\gg_x 1$}
.$$
\item If   there exists $C>1$ such that $\bms(\Psi_{2^j})\leq C\cdot  \bms(\Psi_{2^{j+1}})$ for all $j\gg 1$ and $\sum_j \bms(\cC_{2^{j+1},\Psi_{2^j}})<\infty,$
then for $\bms$-a.e. $x\in X$, 
$$\lambda_t(\Psi_t)(x)\leq  (4C )\cdot \bms(\Psi_t)\quad\text{ for all $t\gg_x 1$}.$$
\end{enumerate}
\end{lem}

\subsection{Hitting along a subsequence}\label{regg}
In the rest of this section, let $\mathcal B=\{B_t\}$ be a family of shrinking targets in $X$.
Recall that a family $\cB$ is inner regular (resp. outer regular) 
if there exist $c>0,\alpha>0$ and smooth positive functions $0\leq \Psi_t^-\leq \Id_{B_t}$ (resp. $\Id_{B_t}\leq \Psi_t^+\leq c$) such that
\begin{itemize}
\item
$\bms(B_t)\leq c\cdot \bms(\Psi_t^-)$ (resp. $\bms(\Psi_t^+)\leq c\cdot \bms(B_t)$);
\item $\cS(\Psi_t^\pm)\leq c\cdot
\bms(B_t)^{-\alpha}$.  
\end{itemize}
A family $\cB$ is regular if it is  inner and outer regular. When we want to emphasize the parameters $c$ and $\alpha$, we say that a family is $(c,\alpha)$-regular.
 Our first application of the effective mean ergodic theorem is the following.
 
\begin{prop}\label{p:HSS}
 Assume that $\mathcal B$ is inner regular and satisfies 
$$\liminf_{t\to\infty}\tfrac{|\log(\bms(B_{t}))|}{ t \bms(B_{t})}=0.$$ Then there is a subsequence $t_j\to \infty$ such that for $\bms$-a.e. $x\in X$,
$$\lambda(\{t\leq t_j: xa_t\in B_{t_j}\})\gg t_j\bms(B_{t_j}) .$$
If $\mathcal B$ is also outer regular, then  for $\bms$-a.e. $x\in X$,
$$\lambda(\{t\leq t_j: xa_t\in B_{t_j}\})\asymp t_j\bms(B_{t_j}).$$
\end{prop}
\begin{proof}
Since $\mathcal B$ is inner regular, there are functions $\Psi_t\in  C^\infty(X)$ with $0\leq \Psi_{t}\leq \Id_{B_t}$ such that $\log(\cS^*(\Psi_t))\ll \log(\bms(B_t))$ and $\bms(\Psi_t)\gg\bms(B_t)$. The mean ergodic theorem (\thmref{mean}) applied to $\Psi_{t}$ implies that 
$$\|\lambda_{t}(\Psi_t)-\bms(\Psi_t)\|^2\ll \frac{(1+\log(\cS^*(\Psi_t)))\|\Psi_t\|^2}{t}.$$

Set $\tilde{\Psi}_t:=\frac{\Psi_t}{\bms(\Psi_t)}$ to get that
$$\|\lambda_t(\tilde \Psi_t)-1\|^2\ll \frac{(1+\log({\cS^*(\Psi_t)}))\|\Psi_t\|^2}{\bms(\Psi_t)^2 \cdot t}\ll \frac{|\log(\bms(B_t))|}{\bms(B_t) \cdot t},$$
where we used that $\|\Psi_t\|^2\leq \bms(B_t)$.
From our assumption, there is some subsequence $t_j$ such that $ \frac{\log(\bms(B_{t_j}))}{\bms(B_{t_j}) \cdot t_j}\to 0$. Hence  $\lambda_{t_j}(\tilde \Psi_{t_j})\to 1$ in $L^2(\G\bk G,\bms)$ and, after perhaps passing to another subsequence, we get
$\lambda_{t_j}(\tilde \Psi_{t_j})(x)\to 1$ for $\bms$-a.e $x\in X$.
For any $x$ in this full measure subset, the inequality $\lambda_{t_j}(\tilde \Psi_{t_j})(x)\leq \frac{\lambda(\{t\leq t_j: xa_t\in B_{t_j}\})}{t_j\bms(\Psi_{t_j})}$ implies that 
$\lambda(\{t\leq t_j: xa_t\in B_{t_j}\})\gg t_j\bms(B_{t_j})$ as claimed. 
Assuming  that$\{B_t\}$ is also outer regular,  repeating the same argument for functions approximating $\Id_{B_t}$ from above gives the other inequality.
\end{proof}

In particular, taking $\lambda$ to be the Lebesgue measure gives the first part of Theorem \ref{t8}. Moreover, by taking $\lambda$ to be the counting measure, we get the 
following consequence implying a discrete version of Theorem \ref{t7}(1).
\begin{cor}
 If $\mathcal B$ is inner regular and
$\liminf_{t\to\infty}\tfrac{|\log(\bms(B_{t}))|}{\bms(B_{t})t}=0$,
 then
$\{k\in \N :xa_k\in B_k\}$ is unbounded for $\bms$-a.e. $x\in X$.
\end{cor}
\begin{proof}
Applying the above result with $\lambda$ the counting measure shows that for $\bms$-a.e. $x\in X$,
 $$\#\{k\leq t_j: xa_k\in B_{t_j}\}\gg t_j\bms(B_{t_j})\to\infty,$$ along some subsequence $t_j$. Since  $B_{t_j}\subseteq B_k$ for any $k\leq t_j$,
it follows that the subset $\{k:xa_k\in B_k\}$ is unbounded as well.
 \end{proof}

 \subsection{Orbits eventually always hitting}
The results of the previous section allow us to control how orbits hit the shrinking targets along a subsequence of times, however, under the same hypothesis we could also have different subsequences for which this asymptotic fails, and for which the set $\{k\leq k_j: xa_k\in B_{k_j}\}$ may even be empty (see e.g. \cite[Proposition 12]{Kel17}).  A more subtle question is to ask what conditions on the shrinking sets guarantee that the truncated orbit  $\{xa_j:j\leq  k\}$ is eventually always hitting the targets $B_k$, and moreover, how large is their intersection? 
This is the content of the following Theorem \ref{t:thresholdD}, which is a discrete version of  Theorem \ref{t7}(2).

 \begin{thm}\label{t:thresholdD}
Assume  that $\cB$ is inner regular and that
$\sum_{j=1}^\infty \frac{|\log(\bms(B_{t_j}))|}{t_{j-1}\bms(B_{t_j})}<\infty$ for some sequence $t_j\to \infty$.
Then for $\bms$-a.e. $x\in X$ and for all $t\gg_x 1$, we have  $\{k\in \N :k\leq t, \; xa_k\in B_t\}\neq \emptyset.$
\end{thm}

\begin{proof} From the inner regularity, we can find smooth functions $0\leq \Psi_{t}\leq \Id_{B_{t}}$ satisfying $\log(\cS_l(\Psi_t))\ll \log(\bms(B_t))$ and $\bms(B_t)\ll \bms(\Psi_t)$. 
By Proposition \ref{p:mCf}, we can estimate  thatfor any $s,t>1$
$$\bms(\cC_{s,\Psi_t})\ll  \frac{|\log({\cS^*(\Psi_t)})| \cdot \|\Psi_t\|^2}{s \bms(\Psi_t)^2}\ll  \frac{|\log(\bms(B_k))| }{s \bms(B_t)}.$$
Since $\bms(\cC_{t_{j-1},\Psi_{t_j}})\ll  \frac{|\log(\bms(B_{t_j}))| }{t_{j-1} \bms(B_{t_j})}$, we obtain 
$\sum_j \bms(\cC_{t_{j-1},\Psi_{t_j}})<\infty$. Hence by the first part of Lemma \ref{l:Cfk}, we have that for $\bms$-a.e. $x\in X$,  
$\lambda_t \Psi_t(x)\neq 0$ for all sufficiently large $t$. Taking $\lambda$ to be the counting measure on $\N$, this implies that 
$\{k\in \N: k\leq t,\; xa_k\in B_t\}\ne \emptyset$ for all sufficiently large $t$.
\end{proof}

 Theorem \ref{qua} implies Theorem \ref{t8}(2).
\begin{thm}\label{t:asymp}\label{qua}
Assume that $\cB$ is regular and that $\bms(B_{2t})\asymp \bms(B_t)$. If
$\sum_{j=1}^\infty \frac{|\log(\bms(B_{2^j}))|}{2^j\bms(B_{2^j})}<\infty , $ then, for $\bms$-a.e. $x$, and for all  $t\gg_x 1$,
$$\frac{\#\{j \leq t : xa_j\in B_{t}\}}{t}\asymp \frac{\ell \{s \leq t : xa_s\in B_{t}\}}{t}\asymp \bms(B_t).$$
\end{thm}

\begin{proof} 
Let $\Psi_t^\pm$ be functions which approximate $\Id_{B_t}$ from above and below such that
$0\leq \Psi_t^-\leq \Id_{B_t}\leq \Psi_t^+\leq c$,
$\log(\cS_l(\Psi_t^\pm))\ll \log(\bms(B_t))$ and $\bms(\Psi_t^+)\asymp \bms (\Psi_t^-)\asymp \bms(B_t)$.  For each of these functions we can use  Proposition \ref{p:mCf} as before to estimate
$\bms(\cC_{s,\Psi^\pm_t})\ll  \frac{|\log(\bms(B_t))| }{s \bms(B_t)}$.
Taking $s=2^{j\pm1}$ and $t=2^j$, we get that $\sum_j \bms(\cC_{2^{j\pm 1},\Psi^\pm_{2^j}})<\infty$. So by the second and third part of Lemma \ref{l:Cfk} we get that 
for $\bms$-a.e. $x\in X$ and for all sufficiently large $t$, we have
$$\bms(B_t)\ll \bms(\Psi_t^-)\ll \lambda_t\Psi_t^-\leq \lambda_t(\Id_{B_t})\leq  \lambda_t\Psi_t^+\ll \bms (\Psi_t^+)\ll \bms(B_t).$$
This implies that $\lambda_t(\Id_{B_t})\asymp \bms(B_t)$. 
Finally, taking $\lambda$ to be the counting measure on $\N$ (resp. the Lebesgue measure)
gives the result for discrete (resp. continuous) time flow.
\end{proof}
\subsection{Logarithm law for the first hitting time}
Using similar arguments utilizing the effective mean ergodic theorem, we can prove the logarithm law for the first hitting time for the discrete flow.
Recall the discrete first hitting time function
\begin{equation}\label{e:hittingtime}
\tau^d_{B}(x)=\min\{k\in \N: xa_k\in B\}. 
\end{equation}
\begin{thm}\label{t:loglaw}
If $\mathcal B$ is inner regular, 
then
$$\lim_{t\to \infty} \frac{\log(\tau^d_{B_t}(x))}{-\log(\bms(B_t))}=1 \quad \mbox{ for $\bms$-a.e. $x\in X$}.$$
\end{thm}

\begin{proof}
We first note that the bound 
$$\liminf_{t\to\infty}\frac{\log(\tau^d_{B_t}(x))}{-\log(\bms(B_t))}\geq 1,$$
holds for $\bms$-a.e.$x$; indeed, this holds in general for any monotone sequence of shrinking targets in a measure preserving dynamical system (see \cite[Lemma 2.2]{KelmerYu17}). It is thus sufficient to show that for $\bms$-a.e. $x$, 
$$\limsup_{t\to\infty}\frac{\log(\tau^d_{B_t}(x))}{-\log(\bms(B_t))}\leq 1.$$

Fix a small $\epsilon>0$ and set
$$\cA_\e^+:=\{x\in X: \limsup_{t\to\infty}\frac{\log(\tau^d_{B_t}(x))}{-\log\bms(B_t)}> 1+2\e \}.$$
Note that if $x\in \cA_\e^+$, then there are arbitrarily large values of $t$ for which  $\tau^d_{B_t}(x)\geq \frac{1}{\bms(B_t)^{1+2\e}}$,
and hence $x\in \cC^o_{k_\e(t),\Psi_t}$ where 
 $\Psi_t=\Id_{B_t}$ and 
$$k_\e(t)=\lfloor\frac{1}{(\bms(B_t))^{1+ 2\e}}\rfloor.$$

Now for any $j\in \N$, we choose $y_j\in (\frac{1}{2^{j+1}},\frac{1}{2^{j}}]$ such that either 
$t_j=\sup\{t:\bms(B_t)\geq y_j\}$ satisfies $\bms(B_{t_j})=y_j$ or there is no $t$ with $\bms(B_t)\in [y_{j},y_{j-1})$ (if the function $t\mapsto \m(B_t)$ is continuous, we may simply take $y_j=2^{-j}$. In general, since the function $t\mapsto \bms(B_t)$ is monotone decreasing, it has at most countably many points of discontinuity and hence we can always find such points). We partition $[0,\infty)$ into intervals $I_j=\{t: \bms(B_t)\in [y_{j+1},y_{j})\}$ and write
$$\cA_\e^+\subseteq \bigcap_{k\in \N}\bigcup_{j>k}\bigcup_{t\in I_j}\cC^o_{k_\e(t),\Psi_t}.$$
For all sufficiently large $j$ and  any $t$ with $\bms(B_t)\in [y_{j+1},y_{j})$, we have that $k_\e(t)\in [2^{(1+\e)( j)},2^{(1+2\e)(j+2)}]$ so that 
$\cC^o_{k_\e(t),\Psi_t}\subseteq \cC^o_{2^{(1+\e) j},\Psi_t}$. Since $B_{t_j}\subseteq B_t$ for all $t<t_j$, we get
$ \cC^o_{2^{(1+\e) j},\Psi_{t}}\subseteq  \cC^o_{2^{(1+\e) j},\Psi_{t_j}}$. We can thus further bound
$$\cA_\e^+\subseteq \bigcap_{k\in \N}{\bigcup_{j>k, I_j\ne \emptyset}}\cC^o_{2^{(1+\e) j},\Psi_{t_j}}.$$

From our choice of $y_j$ and $t_j$, we have that $\bms(\Psi_{t_j})=y_j\in  (\frac{1}{2^{j+1}},\frac{1}{2^{j}}]$.
Since $\{B_t\}$ is inner regular, we have $0\leq \Psi_{t_j}^{-}\leq \Psi_{t_j}$ with $\bms(\Psi_{t_j}^-)\asymp \bms(\Psi_{t_j})$ and $\log(\cS^*(\Psi_{t_j}^-))\ll |\log(\bms(\Psi_{t_j}))|\ll j$. Using Proposition \ref{p:mCf} for the smooth functions as before, we bound 
$$\bms(\cC_{2^{j(1+\e)},\Psi_{t_j}})\leq \bms (\cC_{2^{j(1+\e)},\Psi^-_{t_j}})\ll \frac{j}{2^{j(1+\e)}2^{-j}}\ll j2^{-\e j}.$$
Hence
$\bms(\cA_\e^+)\leq \sum_{j>k}j2^{-\e j}\ll k2^{-\e k}$ for all $k\in \N$. Therefore $\bms(\cA_\e^+)=0$ and 
$$\limsup_{t\to\infty}\frac{\log(\tau^d_{B_t}(x))}{-\log\bms(B_t)}\leq 1+2\e \quad\text{for $\bms$-a.e. $x\in X$.}$$  This holds for any $\e>0$.
 Hence, by taking a sequence of $\e_j\to 0$, we finish the proof.  
\end{proof}

\subsection{Thickening along the flow.} \label{s:cont}
We note that if  $\{k\in \N:xa_k\in B_k\}$ is unbounded (resp. $\{j\leq k: xa_j\in B_k\}\ne \emptyset$), then 
 $\{t\in\R: xa_t\in B_t\}$ is unbounded (resp. $\{t\leq k: xa_t\in B_k\}\ne \emptyset$). Hence the same assumptions on the shrinking rate of $\bms(B_t)$ as in Proposition \ref{p:HSS} give the same conclusions also for the continuous flow. However, it is possible for the set $\{t\in (0, \infty): xa_t\in B_t\}$ to be unbounded even when it is bounded for the discrete time flow. In order to get the correct thresholds for the  continuous flow, one needs to consider the thickened targets.

For any set $B\subseteq X$ we define its thickening $\tilde{B}$ to be
\begin{equation}\label{e:think}
\tilde{B}=\bigcup_{|s|<1/2} Ba_s.
\end{equation}
In the following lemma we observe that the shrinking target problems for the continuous flow can be translated to similar problems for the discrete flow hitting the thickened targets. 
\begin{lem}
For any $B\subseteq X$ and $x\in X$, we have:
\begin{enumerate}
\item If $xa_t\in B$ for some $t\in \R$, then $xa_k\in \tilde{B}$ for $k\in \Z$ with $|t-k|\leq 1/2$.
\item If $xa_k\in \tilde{B}$ with $k\in \Z$, then $xa_t\in B$ for some $t$ with $|t-k|\leq 1/2$.
\item $|\tau_B(x)-\tau^d_{\tilde{B}}(x)|\leq 1/2 .$
\end{enumerate}
\end{lem}
 The proof of these observations is easy once stated and we omit the details. Using this, we get the following sharper results for the continuous time flow, which imply Theorems \ref{t6} and \ref{t7}.
 \begin{thm}\label{t:CST}\label{inl}
 Suppose that the family $\{\tilde{B}_t\}_{t\geq 1}$ of  thickened targets  is inner regular.
 \begin{enumerate}
 \item If  $\liminf_{k\to \infty} \frac{|\log(\bms(\tilde{B}_k)|}{\bms(\tilde{B}_k) k}= 0$
  then for $\bms$-a.e. $x\in X$, 
 $\{t\in \R: xa_t\in B_t\}$ is unbounded.
 \item If 
$\sum_{j=1}^\infty \frac{|\log(\bms(\tilde{B}_{2^j}))|}{2^j\bms(\tilde{B}_{2^j})}<\infty$,
then for $\bms$-a.e. $x\in X$, 
 $$\{0<s<t: xa_s\in  B_t\} \neq \emptyset\quad\text{ for all $t\gg_x 1$} .$$ 
\item For $\bms$-a.e. $x\in X$,
$$\lim_{t\to\infty} \frac{\log\tau_{B_t}(x)}{-\log\bms(\tilde{B}_t)}=1.$$
 \end{enumerate}
 \end{thm}
 \begin{proof}
 The first condition (with $k$ replaced by $k+1$) implies that the set $\{k\in \N: xa_k\in \tilde{B}_{k+1}\}$ is unbounded. For each $k$ in this set, there is some $t_k\in [k-1/2,k+1/2]$ with $xa_{t_k}\in B_{k+1}\subseteq B_{t_k}$, proving the first part.
 
For the second part, the summability condition implies that for $\bms$-a.e. $x$, we have that $\{xa_j:j\leq k)\}\cap \tilde{B}_k\neq \emptyset$ for all sufficiently large $k>k_0$. Now for $t\geq k_0+1$ and $k:=\lfloor t\rfloor$, there is some $j\leq k$ with $xa_j\in \tilde{B}_k$; hence there is $s\leq t$ with $xa_s\in B_k\subseteq B_t$.

Finally for the last part, since $|\tau_B(x)-\tau^d_{\tilde{B}}(x)|\leq 1/2$, we get that 
$$\lim_{t\to\infty} \frac{\log\tau_{B_t}(x)}{-\log \bms(\tilde{B}_t)}=\lim_{t\to\infty} \frac{\log\tau^d_{\tilde{B}_t}(x)}{-\log \bms(\tilde{B}_t)}.$$
 \end{proof}
 
\begin{rem}
 The problem of estimating $\Leb \{ t\le k: xa_t\in B_k\}$, for the continuous time flow,
 does not easily reduce to the discrete time problem for the thickened targets. Here, knowing that $xa_k\in \tilde{B_k}$ only tells us that $xa_t\in B_k$ for some $t$ close to $k$ but not on the amount of time spent there. Hence, to get asymptotics we need the stronger condition that $\sum_{j=1}^\infty \frac{|\log(\bms(B_{2^j}))|}{2^j\bms(B_{2^j})}<\infty$ for the original sets and not the thickened sets. In particular, if $\bms(B_k)\asymp {k^{-a}}$ for some $a\geq1$ and $\bms(\tilde{B}_k)\asymp {k^{-b}}$ for some $b<1$, then by reducing to the thickened case, we know that for all sufficiently large $k$,  $\{t\leq k: xa_t\in B_k\}\ne \emptyset$, but we do not get an asymptotic estimate for the size of these sets.
  \end{rem}

 \section{Explicit examples}\label{s:example}
 In this section, we consider explicit examples of shrinking targets given by shrinking cusp neighborhoods, shrinking metric balls and shrinking tubular neighborhoods, 
 and show that they are regular and approximate their measure. 

 \subsection{Cusp neighborhoods}\label{s:cusp}
 Let $\fh_{1}, \cdots, \fh_k$ and $\fh_{i,t}$ be the cusp neighborhoods defined in \eqref{hii}.
 In order to apply our results for these sets we need to verify that the family $\{\fh_{i,t}\}_{t\geq 1}$ is regular and satisfies  
 $\m(\fh_{i,t})\asymp e^{-t(2\delta-\kappa_i)}$ where $\kappa_i$ is the rank of the parabolic fixed points associated to $\fh_{i}$. 
 While the upper bound $\m(\fh_{i,t})\ll e^{-t(2\delta-\kappa_i)}$ is proved in \cite{DOP} and \cite{Roblin05}, we could not find a reference where the lower bound is established; so we include a proof for the convenience of readers. 
 
 The important feature of a geometrically finite group is that all of its parabolic fixed points are bounded, i.e.,
 the stabilizer of $\xi$ in $\Gamma$ acts cocompactly on $\Lambda-\{\xi\}$ for each parabolic fixed point $\xi$. This is the main ingredient
 of the argument below. We refer to  \cite{Bo} for the description of horoballs in geometrically finite manifolds that will be used below.

We will work here with the upper half space model
 $$\bH^n=\{z=(x,y): x\in \R^{n-1},\;y>0\},$$
 and  fix our base point to be $o=(0,1)$. Since we will work with one fixed cusp, we may assume without loss of generality that it is the infiniy $\infty$. 
  Set $\Gamma_\infty:=\op{Stab}_\Gamma \infty$ and $\kappa$ to be the rank of $\infty$.
   Without loss of generality, we assume that $\Gamma_\infty=\z^\kappa$. Fix  a horoball $\tilde H(0)\subset \bH^n$ 
such that 
$$\G_\infty=\{\gamma\in \Gamma: \tilde H(0) \cap \gamma \tilde H(0) \ne \emptyset\}=\{\gamma\in \Gamma: \tilde H(0) = \gamma \tilde H(0) \} .$$
  In fact, $\tilde H(0)$ is of the form $\{(x,y): y=y_0\}$ for some $y_0>0$. For the notational simplicity, we assume $y_0=1$.
  Set $\tilde H(t)=\{z \in \tilde H(0): d(z, \partial \tilde H(0) )\ge t\}=\{(x,y): y\ge e^t\}$.
  Without loss of generality, we may assume $\pi(\fh_t)=\Gamma_\infty\ba \tilde H(t)$ where  $\fh_{\infty,t}=\fh_t$.

 Choose a fundamental domain $\cF_\infty\subseteq \R^{n-1}$ for the action of $\G_\infty$ on $\br^{n-1}$ containing the origin so that the sets,
$\op{int}(\gamma \cF_\infty)$, are mutually disjoint for $\gamma\in \Gamma_\infty$.  Note that $H'(t)=\{z=(x,y): x\in\cF_\infty: y\geq e^t\}$ is a fundamental domain for $\pi(\fh_t)=\G_\infty\bk \tilde H(t)$.
We  can choose a  compact fundamental parallelepiped  $ \mathcal P$ containing $\cF_\infty\cap \Lambda$ such that $\G_\infty \cP$ covers $\Lambda\setminus\{\infty\}$ and $\mathrm{int}(\g\cP)$s are mutually disjoint for all $\g\in \G_\infty$.  We may choose $\cP$ to contain the origin so that if $H(t):=H'(t)\cap \mathrm{hull}(\Lambda)$,
then \begin{equation}\label{inj} (\cF_\infty\cap \Lambda) \times [e^t, \infty) \subset H(t)\subset \mathcal P\times [e^t, \infty) .\end{equation}

 As $\cP$ is compact, we have for any $z\in H(t)$, we have $d(\G o, z)=d(o,z)$, and for $z\in \partial H(t)$,
$$d(\G o, z)=d(o,z)=t+O(e^{-t}).$$  
The following is also clear from \eqref{inj}:
 \begin{prop}\label{CuspInj}
The injectivity radius $r_z$ at any point $z\in \partial H(t)$ satisfies $r_z\asymp e^{-t}$, where the implied constants are uniform for all $t\gg 1$.
\end{prop}
 
We will use the following well-known fact:
\begin{prop}\label{p:Ht}
There exists $c>0$ such that for all $t\geq 0$,
$$H(t+c)\subset \{z\in \mathrm{hull}(\Lambda)\cap H(0) :  d(z, \G o)\geq t\} \subseteq  H(t-c).$$
\end{prop}

Next we want to estimate the measure $\m(\fh_t)$ for large $t$. For any $\xi^-\ne \xi^+\in \partial \bH^n-\{\infty\}$ and $s\in \R$, we denote by $\xi_s$ the unit speed geodesic  from $\xi^-$ to $\xi^+$ (where $s$ is the signed distance from the highest point of the geodesic), and recall that this gives us the coordinates $(\xi^-,\xi^+,s)$ parametrizing $\T^1(\cM)$. Let $\Lambda'=\Lambda\setminus \{\infty\}$ and let $\cP_0=\cF_\infty\cap \Lambda$. 

We first show the following:
\begin{lem}\label{l:CP}
$$\m(\fh_t)=\sum_{\g\in \G_\infty}\int_{\cP_0}\int_{\g\cP_0}\int_\R \Id_{\tilde{H}(t)}(\pi(\xi_s))d\m(\xi^-,\xi^+,s).$$
\end{lem}
\begin{proof}
Let $\cF_\G\subseteq \bH^n$ be a fundamental domain for $\G\bk \bH^n$ containing $o$ such that for $t\geq  0$ sufficiently large, we have that $\cF_\G\cap \tilde{H}(t)=H'(t)$,
so that $\m(\fh_t)=\int_{\T^1(\cM)}\Id_{H'(t)}d\m$. 
Since $\{(\xi^-,\xi^+,s): \{\xi^\pm\}\cap \{\infty\}\ne \emptyset \}$ has $\m$-measure zero, we can rewrite this in the $(\xi^-,\xi^+,s)$ coordinates as 
$$\int_{\T^1(\cM)}\Id_{H'(t)}d\m=\int_{\Lambda'}\int_{\Lambda'}\int_\R \Id_{H'(t)}(\pi(\xi_s))d\m(\xi^-,\xi^+,s).$$
Now decomposing $\Lambda'$ as a union over translates $\g\cP_0$ with $\g\in \G_\infty$, we can rewrite
\begin{eqnarray*}
\m(\fh_t)&=&\sum_{\gamma,\gamma'\in \G_\infty} \int_{\g \cP_0}\int_{\g'\cP_0}\int_\R\Id_{H'(t)}(\pi(\xi_s))d\m(\xi^-,\xi^+,s)\\
&=&\sum_{\gamma,\gamma'\in \G_\infty} \int_{\cP_0}\int_{\g'\cP_0}\int_\R\Id_{\g^{-1}H'(t)}(\pi(\xi_s))d\m(\xi^-,\xi^+,s)\\
&=&\sum_{\gamma\in \G_\infty} \int_{\cP_0}\int_{\g\cP_0}\int_\R\Id_{\tilde{H}(t)}(\pi(\xi_s))d\m(\xi^-,\xi^+,s)\\
\end{eqnarray*}
where for the second line we made a change of variables $\xi\mapsto \g\xi$ and in the  last line we used that $\tilde{H}(t)=\bigcup_{\g\in \G_\infty} \g H'(t)$.
\end{proof}

In order to evaluate this, we need the following geometric estimate.
\begin{lem}\label{l:TIC}
Let $\xi^-\in \cP_0$ and $\xi^+\in \g \cP_0$ with $\g\in \G_\infty$. Then there exists $c>0$ such that 
\begin{equation}\int_\R\Id_{\tilde{H}(t)}(\pi(\xi_s))ds=\begin{cases} d(o,\g o)-2t+O(1)&\text{
if $d(o,\g o)>2t-c$}\\ 0&\text{ otherwise}.\end{cases}\end{equation}
\end{lem}
\begin{proof}
Recall that $o=(0,1)$ and note that $\g o=(v,1)$ for some $v\in \R^{n-1}$. Since $ \cP_0$ is a compact set containing the origin, then $\g \cP_0$ is a compact set (of the same diameter) containing $v$ and hence  $\|\xi^- - \xi^+\|=\|v\|+O(1)$ where $\|\cdot\|$ is the Euclidian norm on $\R^{n-1}$. Note that $\sup\{t: \xi\cap H(t)\neq \emptyset\}=\log(\tfrac{\|\xi^--\xi^+\|}{2})$ and $d(o,\g o)=\log(\|v\|)+O(1)$. Hence  if $d(o,\g o)<2t-c$, then $\xi\cap H(t)=\emptyset$.

Now assume that $\xi\cap H(t)\neq \emptyset$ and let $z_1,z_4\in \bH^n$ be the first and second intersections of the geodesic $\xi_s$ with $\partial \tilde H(0)$ and $z_2,z_3$ the first and second intersections with $\partial H(t)$.  Writing $z_i=(x_i,y_i)$, we have that $\|x_1\|$ and $\|x_4-v\|$ are uniformly bounded and that $\|x_2\|$ and $\|x_3-v\|$  are bounded by $O(e^t)$; this implies that 
$d(z_1,o)$, $d(z_4,\g o)$, $d(z_2, a_t o)$ and $d(z_3,\g a_t o)$ are all uniformly bounded. Now on one hand, $d(z_1,z_4)=d(o,\g o)+O(1)$, and on the other hand, since $z_1,z_2,z_3,z_4$ all lie on the same geodesic, we have
$d(z_1,z_4)=d(z_1,z_2)+d(z_2,z_3)+d(z_3,z_4)$. The middle term is precisely $\int_\R\Id_{\tilde{H}(t)}(\pi(\xi_s))ds$ and
$d(z_1,z_2)=d(o, a_t o)+O(1)=t+O(1)$  and similarly $d(z_3,z_4)=t+O(1)$, concluding the proof. \end{proof}

Recall the notation $\tilde\fh_t=\cup_{|s|<1/2}\mathcal G^s \fh_t$.
\begin{prop}\label{p:BMScusp}
We have $\m(\fh_t)\asymp \m(\tilde \fh_t)\asymp  e^{-t(2\delta-\kappa)}$.
\end{prop}
\begin{proof}
From Lemma \ref{l:CP}, we have 
\begin{eqnarray*}
&&\m(\fh_t)=\sum_{\g\in \G_\infty}\int_{\cF_\infty}\int_{\g\cF_\infty}\int_\R \Id_{\tilde{H}(t)}(\pi(\xi_s))d\m(\xi^-,\xi^+,s)=\\
&&\sum_{\g\in \G_\infty}\int_{\cP_0}\int_{\g\cP_0}\int_\R \Id_{\tilde{H}(t)}(\pi(\xi_s))e^{\delta(\beta_{\xi^+}(o,\pi(\xi_s))+\beta_{\xi^-}(o,\pi(\xi_s)))}d\nu_o(\xi^-)d
\nu_o(\xi^+)ds
\end{eqnarray*}
Next note that for any $\xi^-\in \cP_0$ and $\xi^+\in \g \cP_0$ and $\g\in \G_\infty$, the sum
$\beta_{\xi^+}(o,\pi(\xi_s))+\beta_{\xi^-}(o,\pi(\xi_s))$ is independent of $s$ and is uniformly bounded.
Indeed, let $s_1$ be the least time such that $z_1:=\pi(\xi_{s_i})\in H(0)$ and note that $d(z_1,o)=O(1)$  is uniformly bounded.  Now, for $z=\pi(\xi_s)$, on one hand 
$\beta_{\xi^+}(z_1,z)+\beta_{\xi^-}(z_1,z)=s-s_1+s_1-s=0$,
and on the other hand 
$|\beta_{\xi^\pm}(z_1,z)-\beta_{\xi^\pm}(o,z)|\leq d(z_1,o)$  which is uniformly bounded.

With this observation together with Lemma \ref{l:TIC}, we get that 
\begin{eqnarray*}
\m(\fh_t)&\asymp &\sum_{\g\in \G_\infty}\nu_o(\cP_0)\nu_o(\g \cP_0)\int_\R \Id_{\tilde{H}(t)}(\pi(\xi_s))ds\\
&\asymp& \mathop{\sum_{\g\in \G_\infty}}_{d(o,\g o)\geq 2t-c} \nu_o(\cP_0)\nu_o(\g \cP_0)(d(o,\g o)-2t+O(1))
\end{eqnarray*}
Next, to estimate $\nu_o(\g \cP_0)=\nu_{\g o}(\cP_0)$, we use the $\Gamma$-conformality to get that
$$\nu_o(\g \cP_0)=\int_{\cP_0}e^{-\delta \beta_\xi(\g o, o)}d\nu_o(\xi).$$ 
To estimate $\beta_\xi(\g o, o)$,  let $z_1,z_2$ be the two points in the intersection of $\partial H(0)$ and the geodesic
connecting $\xi$ to $\g \xi$.  Then $d(z_1,o)$ and $d(z_2,\g o)$ are uniformly bounded and $\beta_\xi(z_1,z_2)=d(z_1,z_2)=d(\g o, o)+O(1)$ implying that $\beta_\xi(\g o, o)=d(\g o,o)+O(1)$.
Plugging in this estimate gives
$$\m(\fh_t)\asymp \mathop{\sum_{\g\in \G_\infty}}_{d(o,\g o)\geq 2t-c} e^{-\delta d(o,\g o)}(d(o,\g o)-2t+O(1)).$$

We may write $\Gamma_\infty$ as $\{ \gamma_v: v\in \z^{\kappa}\}$ where $\gamma_v$ is the translation by $v$.
Note that $ d(o, \g_v(o))=2 \log\|v\|+O(1)$. Hence\begin{eqnarray*}
\m(\fh_t)&\asymp&\mathop{\sum_{v\in \Z^\kappa}}_{\|v\|\geq Ce^t} e^{-\delta (2\log\|v\|+O(1))}(2\log\|v\|-2t+O(1))\\
&\asymp &\mathop{\sum_{v\in \Z^\kappa}}_{\|v\|\geq Ce^t} \|v\|^{-2\delta}\log(\|v\|e^{-t})\\
&\asymp &\int_{x\in \br^k, |x|\ge e^t} \|x\|^{-2\delta}\log(\|x\|e^{-t})dx\asymp e^{-t(2\delta-\kappa)}\\
\end{eqnarray*}
as claimed.

For the thickened target, for any $x\in \fh_t$ and $|s|\leq 1/2$,  if $xa_t\in \fh_{t-1}$, then $\fh_t\subseteq \tilde\fh_t\subseteq \fh_{t-1}$. Hence 
$\m(\tilde\fh_t)\asymp e^{-t(2\delta-\kappa)}$ as well.
\end{proof}

Next we show regularity.
\begin{prop}
Both families
$\{\fh_t: t\geq 1\}$ and $\{\tilde \fh_t: t\geq 1\}$ are regular.
\end{prop}
\begin{proof} 
Since $\fh_t\subseteq \tilde\fh_t\subseteq \fh_{t-1}$ it is enough to show that $\{\fh_t\}$ is regular.
Let $H'(t)$ denote the fundamental domain for $\G_\infty\bk \tilde{H}(t)$ defined above, and $\cF_\G$ a fundamental domain for $\G\bk \bH^n$ such that $\cF_\G\cap \tilde{H}(t)=H'(t)$. For any $t\geq 1$ let $\psi_t^\pm$ be smooth functions on $\G_\infty\bk \tilde{H}(t)$ taking values in $[0,1]$ satisfying 
$$\Id_{H'(t+1)}\leq \psi_t^-\leq \Id_{H'(t)}\leq \psi_t^+\leq \Id_{H'(t-1)},$$ 
and we can choose them so that $\cS(\psi_t^\pm)=O(1)$, independent of $t$.

Since $\cF_\G\cap \tilde{H}(t)=H'(t)$, we can lift the functions $\psi_t^\pm$ to right $K$-invariant, and  left $\G$-invariant functions $\Psi_t^\pm$ on $G$. As such, by looking at their values on a fixed fundamental domain, we see that
$$0\leq \Id_{\fh_{t+1}}\leq \Psi_t^-\leq \Id_{\fh_t}\leq \Psi_t^+ \leq \Id_{\fh_{t-1}}\leq 1.$$ 
Since $\m(\fh_t)\asymp \m(\fh_{t\pm1})$, we also get that $\m(\Psi_t^\pm)\asymp \m(\fh_t)$, implying that $\{\fh_t\}$ is regular. 
\end{proof}

\begin{proof}[Proof of Theorem \ref{t1}]
Applying Theorem \ref{t6} to the shrinking targets $B_t=\fh_{i,t}$ gives (1). 
For (2), fix some $\eta<\frac{1}{2\delta-\kappa_i}$ and let $c:=1-\eta(2\delta-\kappa_j)>0$.  Consider the shrinking family $\{B_t=\fh_{i,\eta\log(t)}\}$, which is regular and satisfies $\m(B_{2t})\asymp \m(B_t) \asymp t^{-(2\delta-\kappa_i)\eta}$. In particular we have that 
 $$\sum_j \frac{\log(\m(B_{2^j}))}{2^j \m(B_{2^j})}\asymp \sum_j \frac{\log(j)}{2^{cj}}<\infty,$$
so Theorem \ref{t8}(2) implies Theorem \ref{t1}(2).  
\end{proof}

 \subsection{Shrinking balls in $\G\bk G$}\label{sball}
In this subsection, our goal is to show that for $x\in \op{supp}(\m)$, the family $\{xG_\e: 0<\e<1\}$ is regular as stated in Proposition \ref{p:GeRegular}.
We may assume that $o\in \op{hull} \Lambda$ and fix $v_o\in \T_o(\bH^n)$ so that $M=\op{Stab}(v_o)$.
 For $\xi\in \partial\bH^n$ and $\e>0$, let $B_\xi(\e)$ denote the Euclidian ball of radius $\e$ around $\xi$. 
When $\G$ is convex co-compact, Sullivan's shadow lemma implies that $\nu_o(B_\xi(\e))\asymp \e^\delta$, but when $\G$ has cusps,
 the measure $\nu_o(B_\xi(\e))$ fluctuates as $\e\to 0$. 
Nevertheless, we have the following:
 \begin{lem}\label{l:Shadow}
 For any $\xi\in \Lambda$, the following holds for all  sufficiently small $\e>0$:
 \begin{enumerate}
 \item  $\min\{\e^\delta, \e^{2\delta-\kappa_{\rm min}}\}\ll\nu_o(B_\xi(\e))\ll \max\{\e^\delta, \e^{2\delta-\kappa_{\rm max}}\}.$
  \item $\nu_o(B_\xi(2\e)) \ll \nu_o(B_\xi(\e))$.
 \end{enumerate}
 \end{lem}
 \begin{proof}  
 For $\xi\in \Lambda$ let $\xi_t\subset \op{hull}(\Lambda)$ denote the unit speed geodesic ray connecting $o$ to $\xi$.
Let $b(\xi_t)\subset \partial (\bH^n)$ denote the shadow at infinity of the hyperbolic hyperplane meeting $\xi_t$ orthogonally. 
 Then 
 $$b(\xi_t)=B_\xi(\e)\quad\text{ for $\e\asymp e^{-t}$.}$$
 If $\{\xi_1, \cdots, \xi_k\}$ denotes the set
of all representatives of $\Gamma$-orbits in the set of parabolic limit points, and $H_{\xi_i}\subset \bH^n$ is a sufficiently deep horoball based at $\xi_i$,
then $H:=\bigcup_{i=1}^k \Gamma (H_{\xi_i})$ forms a family of disjoint horoballs.

Now by \cite[Theorem 2]{SV95}, we have  \begin{equation}\label{e:shadow}
 \nu_o(b(\xi_t))\asymp e^{-\delta t+d(\xi_t,\G o)(\kappa(\xi_t)-\delta)},
 \end{equation} 
 where $\kappa(\xi_t)$ is the rank of $\xi_i$ if $\xi_t\in \Gamma (H_{\xi_i})$ for some $i$, and $\kappa(\xi_t)=\delta$ otherwise. 
Now, to prove (1), let $\e=e^{-t}$. First, if $\xi_t$ is not in  $H$, the claim follows easily.
 Next, if $\xi_t\in \Gamma(H_{\xi_i})$, then $\kappa(\xi_t)=\kappa_i$ and $\nu_o(b(\xi_t))\asymp e^{-\delta t+d(\xi_t,\G o)(\delta-\kappa_i))}$.
If $\kappa_i\le \delta$, then
 $$\delta t\leq \delta t+d(\xi_t,\G o)(\delta-\kappa_i)\leq 
  t(2\delta-\kappa_{\rm min})$$
 and hence $e^{-t(2\delta-\kappa_{\rm min})}\ll \nu_o(b(\xi_t))\ll e^{-\delta t}$.
Now if $\kappa_i>\delta$, then
 $$-\delta t\leq -\delta t+d(\xi_t,\G o)(\kappa_i-\delta)\leq  t(\kappa_{\rm max}-2\delta),$$
 so that $e^{-\delta t}\ll \nu_o(b(\xi_t))\ll e^{-(2\delta-\kappa_{\rm max}) t}$. This proves (1).

For (2), we claim that $\nu_o(b(\xi_{t+1}))\asymp \nu_o(b(\xi_t))$.
In the case when $\xi_t,\xi_{t+1}\in \Gamma(H_{\xi_i})$ for some $i$,  we have that $|d(\xi_t,\G o)-d(\xi_{t+1},\G o)|\ll1$. 
Now,  using \eqref{e:shadow} we get
 $$\frac{ \nu_o(b(\xi_{t+1}))}{ \nu_o(b(\xi_t))}\asymp  e^{(d(\xi_{t+1},\G o)-d(\xi_{t},\G o))(\delta-k_i)}\asymp 1.$$
If this case does not happen, there must be some $t'\in [t,t+1]$ such that the projection of $\xi_{t'}$ in $\op{core}(\cM)$
 lies in the compact part  $\op{core}(\cM)- \cup_i \Gamma \ba H_{\xi_i} $, and hence $d(\xi_{t'},\G o)=O(1)$. But then also $d(\xi_{t},\G o)$ and $d(\xi_{t+1},\G o)$ are bounded and 
 $\nu_o(b(\xi_{t+1}))\asymp \nu_o(b(\xi_t))\asymp e^{-\delta t}$ as well.
 \end{proof} 

 \begin{prop}\label{pp1}
Let $\cK\subseteq X$ be a compact subset. Let $\delta_-=\min\{\delta,2\delta-k_{\rm max}\}$ and $\delta_+=\max\{\delta, 2\delta-k_{\rm min}\}$. For any $x\in \cK\cap \Omega$, we have that for all $0<\e<r_x$,
\begin{enumerate}
\item $\e^{1+\dim{M}+2\delta_+} \ll \bms(xG_\e)\ll\e^{1+\dim{M}+2\delta_-},$
\item $\bms(xG_{2\e})\asymp \bms(xG_\e),$
\item $\bms(xG_\e A_1)\asymp  \e^{-1}\bms(xG_\e )$, 
\item $\bms(xG_\e M)\asymp\e^{-\dim(M)}\bms(xG_\e )$, 
\end{enumerate}
where all the implied constants above are uniform over all $x\in \cK$.
\end{prop}
\begin{proof} 
Fix a compact subset $\cF_0\subseteq G$ such that $\cK=\Gamma\ba \Gamma \cF_0$.  First, since we assume $\e\leq r_x$,
 we have that $\bms(xG_\e)=\m(gG_\e)$ for $x=[g]$.
We will use the flow boxes
 \begin{equation}\label{e:fb}
 \cB(g,\e):=g\cB(\e)=g(N_\e^+N^-\cap N_\e^-N^+AM)M_\e A_\e.
 \end{equation}
 It is shown in \cite[Lemma 4.7]{MMO17}  that $\cB(g, \e)\asymp  g G_\e$
 and that 
 \begin{equation}\label{e:BMSfb} \bms(\cB(g,\e))=(1+O(\e))2\e\nu_{g (o)}(gN_\e^+v_o^+)\nu_{g(o)}(gN_\e^-v_o^-)\vol_M(M_\e),\end{equation}
 where $\vol_M(M_\e)\asymp \e^{\dim(M)}$ and all implied constants are absolute.

We can estimate $\nu_{g (o)}(gN_\e^\pm v_o^\pm)\asymp \nu_o(B_{g^\pm}(\e))$, with the implied constants uniform for $g\in \cF_0$.
 Hence by Lemma \ref{l:Shadow}, we have 
$$\e^{\delta_+}\ll\nu_{g (o)}(gN_\e^\pm v_o^\pm)\ll\e^{\delta_-}.$$
Since $\vol_M(M_\e)\asymp \e^{\dim{M}}$, we get that 
$$ \e^{2\delta_+ +1+\dim{M}}\ll \bms(g\cB(\e))\ll \e^{2\delta_- +1+\dim{M}}$$
proving (1).  (2) follows similarly from Lemma \ref{l:Shadow}(2). 
(3) and (4) follow easily from the above description of $g\cB(\e)$.
\end{proof}

 \begin{prop}\label{p:GeRegular}\label{pp2}
 Fix a compact set $\cK\subseteq X$.  There exist some $c>1$ and $\alpha>1$(depending on $\ell $, and $\cK$) such that the family $\{xG_\epsilon:\; x\in \cK\cap \Omega,\; \epsilon<r_x\}$ and the family of their thickenings are regular for $\cS_{l}$.
 \end{prop}

\begin{proof}
We can find smooth functions  $\Psi^\pm_\e:G\to [0,1)$ such that 
$$\Psi_\e^-(g)=\left\{\begin{array}{cc}1 & g\in G_{\e/2}\\ 0 &g\not\in G_\e\end{array}\right\}\quad
\Psi_\e^+(g)=\left\{\begin{array}{cc}1 & g\in G_{\e}\\ 0 & g\not\in G_{2\e}\end{array}\right\}$$
 satisfying $\cS_l(\Psi_\e^\pm)\ll \e^{-l}$. For $x\in \cK\cap \Omega$, let 
$\Psi_{x,\e}^\pm(g):=\Psi_{\e}^\pm(xg)$. Then 
$$0\leq \Psi_{x,\e}^-\leq \Id_{xG_\e}\leq  \Psi_{x,\e}^+.$$
We then have that
$$\cS_l(\Psi_{x,\e})\ll \e^{-l}\ll \bms(G_\e)^{-\alpha},$$
for $\alpha=\frac{l}{1+\dim(M)+\delta_-}$
and that $\bms(xG_{\e/2})\leq \bms(\Psi_{x,\e}^-)\leq  \bms(xG_\e)$ so that  $\bms(xG_{\e})\ll \bms(xG_{\e/2})\leq \bms(\Psi_{x,\e}^-)$, and similarly 
$\bms(\Psi_{x,\e}^+)\ll \bms(xG_\e)$. 

The same argument shows that the thickened sets $xG_\e A_1$ are $(c,\alpha)$-regular for some constant $c>1$ and $\alpha=\frac{l}{\dim(M)+\delta_-}$.
\end{proof} 

The proofs of Propositions \ref{pp1} and \ref{pp2} can easily be adapted for the following:
\begin{prop}\label{equal} Let $\cM$ be convex cocompact. Fix $x_0\in \op{supp}(\m)$.
Then the families $\{x_0G_\e M \}$ and $\{x_0G_\e M A_{1/2} \}$ are regular
and $\m(x_0G_\e M)\asymp \e^{2\delta+1}$ and $\m(x_0G_\e M A_{1/2})\asymp \e^{2\delta}$ with the implied constants uniform over all $x_0$.

\end{prop}

When $\cM$ has cusps, we do not have such asymptotics for $\m(x_0G_\e M)$ and $\m(x_0G_\e M A_{1/2})$ uniformly for all $x_0\in  \op{supp}(\m)$. Nevertheless, we have the following estimates:
\begin{prop}\label{p:balls}
Let $\cK\subseteq X$ be a compact subset of $X$, and let $x_0=[g_0M]\in \cK\cap \Omega$.
\begin{enumerate}
\item If both $g_0^+, g_0^-\in \partial \bH^n$ are parabolic fixed points corresponding to cusps of ranks $\kappa_1$ and $\kappa_2$ respectively, then 
$$\m(x_0 G_\e M)\asymp \e^{4\delta+1-\kappa_1-\kappa_2}.$$
\item If  $x_0A$ is bounded,  then  
 $$\m(x_0G_\e M)\asymp \e^{2\delta+1}.$$
 \item If $\sup_{t\in \br} \frac{d(x_0a_t, \G o)}{\log |t|}<\infty $,  then 
 $$\lim_{\e\to 0}\frac{\log(\m(x_0G_\e M))}{\log\e}=2\delta+1.$$
 \end{enumerate}
\end{prop}
\begin{proof}
Without loss of generality, we may assume that $g_0=e$. Set $\xi_t:=a_t$ and let $\xi^\pm:=\lim_{t\to\pm\infty} \xi_t$. Recall that by  \cite[Lemma 4.7]{MMO17} and  \eqref{e:BMSfb}, we have
$$\m(x_0G_\e M)\asymp \m(\cB(g_0,\e)M)\asymp \e \cdot \nu_o(B_{\xi^+}(\e))\nu_o(B_{\xi^-}(\e)).$$ 
It thus remains to estimate $\nu_o(B_{\xi^\pm}(\e))$ in each of the above cases.
When $\xi^\pm$ are parabolic limit points, there exists $t_0$ such that for all $t\geq t_0$ (resp, $t<-t_0$), we have that $\xi_{\pm t}\in H_{\xi^\pm}$ is in the horoball centered at $\xi^\pm$. Since $\xi^\pm$ are parabolic limit points, this implies that for $t\geq t_0$, we have that $d(\xi_t,\G o)=|t|+O(1)$, and hence, setting $\e=e^{-|t|}$ by
\eqref{e:shadow}, we can estimate $\nu_o(B_{\xi^+}(\e))\asymp \e^{2\delta-\kappa_1}$  and similarly $\nu_o(B_{\xi^-}(\e))\asymp \e^{2\delta-\kappa_2}$, proving (1).

Next, the boundedness of $x_0A$ means that $\sup_td(\xi_t,\G o)<\infty$.
In this case, \eqref{e:shadow} implies that $\nu_o(B_{\xi^\pm}(\e))\asymp \e^{\delta}$, proving (2).

Finally, assuming that  $\sup \frac{d(x_0a_t, \G o)}{\log |t|}<\infty$, again taking $\e=e^{-t}$,  \eqref{e:shadow}  now implies that 
$$\e^\delta |\log(\e)|^{-c_1}\ll  \nu_o(B_{\xi^\pm}(\e))\ll \e^\delta |\log(\e)|^{c_1},$$
and hence 
$$\log(\m(x_0G_\e M))=( 2\delta+1)\log\e+O(\log|\log\e|).$$ 
This proves (3).
 \end{proof}

 We finish this section with the proofs of Theorems \ref{t3} and \ref{t4}.
 \begin{proof}[Proof of Theorem \ref{t3}]
(1) follows by applying Theorem \ref{t6} to the shrinking targets $B_t=x_0 G_{1/t} M$ with thickening $\tilde{B_t}=x_0G_{1/t}MA_{1/2}$ which is inner regular with $\log(\m(\tilde B_t))=-2\delta t+O(1)$ by Proposition \ref{equal}.

 For (2) we consider the shrinking targets $B_t:=x_0G_{t^{-\eta}}M$.
 Note that for any $x\in \cM$, we have that $d(\cG^s(x),x_0)<t^{-\eta}$ exactly when $\cG^s(x)\in B_t$. Since
 $\m(B_t)\asymp t^{-\eta(2\delta+1)}$, we have $\sum_j\frac{\log(\m(B_{2^j})}{2^j\m(B_{2^j})}<\infty$ when $(2\delta+1)\eta<1$. So (2) follows from
 Theorem \ref{t8}(2).
   \end{proof}
 
 \begin{proof}[Proof of Theorem \ref{t4}]
 For (1), we note that Theorem \ref{t1}(1) implies that for $\m$-a.e. $x_0\in \T^1(\cM)$,
 $$\limsup_{t\to\infty}\frac{d(\cG^s(x_0),o)}{\log(t)}\le \frac{1}{2\delta-\kappa_{\rm max}} .$$ For any such $x_0$, the families
 $\{B_t=x_0G_{1/t}M\}$ and $\{\tilde{B_t}\}$ are regular. 
 
 By Proposition \ref{p:balls}(3), we have
 $\lim_{t\to \infty}\frac{-\log(\m( B_t))}{\log t}=2\delta+1$, and hence  $\lim_{t\to \infty}\frac{\log(\m(\tilde B_t))}{-\log t}=2\delta$.
Now, using this limit together with Theorem \ref{t6}, we get that for $\m$-a.e. $x\in \T^1(\cM)$
$$\lim_{t\to\infty}\frac{\log(\tau_{B_t}(x))}{\log(t)}=2\delta \quad \text{ and } \quad \lim_{t\to\infty}\frac{\log(\tau_{B_t}(x))}{-\log(\m(\tilde B_t))}=1.$$

For (2), given two cusps $\xi_1,\xi_2$ with ranks $\kappa_1,\kappa_2$, consider a geodesic connecting $\xi_1$ to $\xi_2$ and let $g_0\in \T^1(\bH^n)$ be any point on this geodesic, and set $x_0=[g_0]\in \T^1(\cM)$. Consider the shrinking targets $B_t=x_0 G_{1/t}M$. By Proposition \ref{p:balls}(1),
 we have  $\log(\m(\tilde{B}_t))=-(4\delta-\kappa_1-\kappa_2)\log(t)+O(1)$ and hence (2) follows from Theorem \ref{t6}. 

 \end{proof}

 \subsection{Shrinking tubular neighborhoods}\label{sg}
For a fixed closed geodesic $\cC\subset \T^1(\cM)$ and $\e>0$, we set 
 $\cC_\e=\{x\in \T^1(\cM): d(\cC, x)<\e\}.$
The proof of Theorem \ref{t55} follows as above from the following.
\begin{prop}
The families  $\{\cC_\e:\e<\e_0\}$ and $\tilde{\cC}_\e=\{xa_s: x\in \cC_\e,\; |s|\leq 1/2\}$  are both regular and satisfy $\bms(\cC_\e)\asymp \bms(\tilde{\cC}_\e) \asymp \e^{2\delta}$.
\end{prop}
\begin{proof}
Recall the notations 
 $X, X(\e)$ and $Y(\e)=X-X(\e)$ from section \ref{s:cusps}.
 They are all $M$-invariant subsets of $\Gamma\ba G$, and in the following proof,
 we will regard them as subsets in $\Gamma\ba G/M$.
We can present $\mathcal C=[g_0] A M/M$ and an element of $\mathcal C$ is
represented by $[g_0]a_tM$ for a unique $0\le t < L$ where $L$ is the length of $\mathcal C$.
Let $\e_0$ be sufficiently small so that  $\cC\subseteq Y({\e_0})$ and let $0<\e\le \e_0$.
 Let $Q_\e$ denote a maximal set of points $x_i\in \cC$ such that the sets $x_iG_\e M$ are pairwise disjoint. Writing $x_i=x_0a_{t_i}$ the condition that $x_i G_\e \cap x_{j}G_\e=\emptyset$ imply that $|t_i-t_{j}|\geq \e$ and the maximality condition implies that $|t_i-t_{i+1}|\leq 3\e$. Hence $\#Q_\e \asymp  L\e^{-1}$. Since
$$\bigcup_{x_i\in Q_\e} x_iG_\e M \subseteq \cC_\e\subseteq \bigcup_{x_i\in Q_\e} x_iG_{3\e}M,$$
we can estimate
$$\sum_{x_i \in Q_\e} \bms( x_iG_{\e} M)\leq \bms(\cC_\e)\leq \sum_{x_i\in Q_\e}\bms( x_iG_{3\e}M).$$

Now Proposition \ref{p:balls}(2) implies that 
$\bms(g_i G_\e M)\asymp \e^{2\delta+1}$ where the implied constant does not depend on $i$.  Summing over all $x_i\in Q_\e$, we get that indeed
$\bms(\cC_\e)\asymp \e^{2\delta}$.

Next, to show regularity, for each point $x_i\in Q_\e$, let $\Psi_{\e,i}^\pm$ be smooth non-negative functions approximating $x_iG_\e M$ from below and $x_iG_{3\e}M$ from above respectively, with $\cS_l(\Psi^\pm_{\e,i})\ll \e^{-l}$, and define $\Psi^\pm_\e=\sum_i \Psi^{\pm}_{\e,i}$. 
Since the sets $x_i G_\e M$ are pairwise disjoint, we have that $\Psi_\e^-\leq \Id_{\cC_\e}\leq \Psi_\e^+$ and moreover
$$\bms(\Psi_\e^+)\leq \sum_{Q_\e} \bms(x_iG_{3\e}M)\ll \sum_{Q_\e} \bms(x_iG_{\e}M)\leq \bms(\cC_\e),$$
and similarly that $\bms(\cC_\e) \ll \bms(\Psi_\e^-)$. Since $\# Q_\e\ll \e^{-1}$, we can bound  $\cS_l(\Psi^{\pm}_\e)\ll \e^{-(l+1)}\ll \bms(\cC_\e)^{-\alpha}$ with $\alpha=\frac{l+1}{2\delta}$, showing that the family $\{\mathcal C_\e\}$ is $(c,\alpha)$ regular for some $c>1$, and $\alpha=\frac{l+1}{2\delta}$.

Finally, note that there is $c\geq 1$ such that $a_{-s} G_\e a_s\subseteq G_{c\e}$ for all $|s|\leq 1/2$. Then any point $x\in \tilde{\cC}_\e$ is of the form $x=x_0a_tga_s M$ with $0\leq t\leq L, g\in G_\e$ and $|s|\leq 1/2$. We can write $ga_s=a_{s}a_{-s}ga_s\in a_sG_{c\e}$, to get that $x\in x_0a_{t+s}G_{c\e}\in \cC_{c\e}$. Therefore $\cC_\e\subseteq \tilde{\cC}_\e\subseteq \cC_{c\e}$, implying that $\{\tilde{\cC}_\e\}$ is also regular with $\bms(\tilde{\cC}_{\e})\asymp \bms(C_\e)$.
\end{proof}

\end{document}